\documentclass{amsart}

\usepackage{amssymb,amsmath,amsthm,amsfonts,framed}
\usepackage{framed,comment,enumerate}
\usepackage{xcolor, framed}
\usepackage{tikz}
\usepackage{color, framed}
\usepackage{graphicx,caption,subcaption}

\usepackage[colorlinks]{hyperref}
\hypersetup{linkcolor=blue,citecolor=blue,filecolor=black,urlcolor=blue}

\usepackage{verbatim}

\usepackage[showonlyrefs]{mathtools}

\usepackage[sort,nocompress]{cite}

\def\R {\mathbb{R}}
\def\C {\mathcal{C}}
\def\D {\mathcal{D}}

\def\Z {\mathbb{Z}}

\def\eps{\varepsilon}
\def\dist{{\rm dist}}
\def\cat{\textrm{cat}}
\def\gen{\textrm{genus}}

\newcommand{\cM}{{\mathcal M}}

\newcommand{\cP}{{\mathcal P}}

\newcommand{\loc}{\mathrm{loc}}

\newcommand{\wc}{\rightharpoonup}

\newcommand{\pa}{\partial}
\newcommand{\mf}[1]{\mathbf{#1}}

\DeclareMathOperator{\supp}{supp}

\DeclareMathOperator{\rad}{rad}

\newtheorem{proposition}{Proposition}[section]
\newtheorem{theorem}[proposition]{Theorem}
\newtheorem*{theorem*}{Theorem}

\newtheorem{lemma}[proposition]{Lemma}

\theoremstyle{definition}
\newtheorem{remark}[proposition]{Remark}
\numberwithin{equation}{section}

\title[Natural constraints for normalized solutions]{A natural constraint approach to normalized solutions of nonlinear Schr\"odinger equations and systems}

\author{Thomas Bartsch and Nicola Soave}

\address{
\hbox{\parbox{5.7in}{\medskip\noindent
Thomas Bartsch\\
Mathematisches Institut, Justus-Liebig-Universit\"at Giessen, \\
Arndtstrasse 2, 35392 Giessen (Germany),\\[2pt]
{\em{E-mail address: }}{\tt Thomas.Bartsch@math.uni-giessen.de.} \\ [5pt]
Nicola Soave\\
Mathematisches Institut, Justus-Liebig-Universit\"at Giessen, \\
Arndtstrasse 2, 35392 Giessen (Germany),\\[2pt]
{\em{E-mail address: }}{\tt nicola.soave@gmail.com, Nicola.Soave@math.uni-giessen.de.}}}}


\thanks{\em{Acknowledgements:} We thank Prof. Louis Jeanjean and Prof. Susanna Terracini for a careful reading of the manuscript and for several precious suggestions. \\
Nicola Soave is partially supported through the project ERC Advanced Grant 2013 n. 339958 ``Complex Patterns for Strongly Interacting Dynamical Systems - COMPAT''.}

\begin{document}

\begin{abstract}
The paper deals with the existence of normalized solutions to the system
\[
\begin{cases}
-\Delta u - \lambda_1 u = \mu_1 u^3+ \beta u v^2 & \text{in $\R^3$} \\
-\Delta v- \lambda_2 v = \mu_2 v^3 +\beta u^2 v & \text{in $\R^3$}\\
\int_{\R^3} u^2 = a_1^2 \quad \text{and} \quad \int_{\R^3} v^2 = a_2^2
\end{cases}
\]
for any $\mu_1,\mu_2,a_1,a_2>0$ and $\beta<0$ prescribed. We present a new approach that is based on the introduction of a natural constraint associated to the problem. We also show that, as $\beta\to-\infty$, phase separation occurs for the solutions that we find.

Our method can be adapted to scalar nonlinear Schr\"odinger equations with normalization constraint, and leads to alternative and simplified proofs to some results already available in the literature.
\end{abstract}

\maketitle

\section{Introduction}

Various physical phenomena, such as the occurrence of phase-separation in Bose-Einstein condensates with multiple states, or the propagation of mutually incoherent wave packets in nonlinear optics, are modeled by the system of coupled nonlinear Schr\"odinger equations
\begin{equation}\label{syst schrod}
\begin{cases}
- \iota \pa_t \Phi_1 = \Delta \Phi_1 + \mu_1 |\Phi_1|^{2} \Phi_1+ \beta
 |\Phi_2|^{2} \Phi_1 \\
- \iota \pa_t \Phi_2 = \Delta \Phi_2 + \mu_2 |\Phi_2|^2 \Phi_2+ \beta |\Phi_1|^{2}\Phi_2
\end{cases} \quad t,x \in \R \times \R^N,
\end{equation}
see e.g.\ \cite{AkAn, Esry, Fra, Mal, Timm}. In the models, $\Phi_i$ is the wave function of the $i$-th component, the dimension of the ambient space is $N \le 3$, and the real parameters $\mu_i$ and $\beta$ represent the intra-spaces and inter-species scattering length, describing respectively the interaction between particles of the same component or of different components. In particular, the positive sign of $\mu_i$ (and of $\beta$) stays for attractive interaction, while the negative sign stays for repulsive interaction.

A fundamental step in the comprehension of the dynamics of the system consists in studying the possible existence and properties of solitary waves, solutions to \eqref{syst schrod} of type $\Phi_i(t,x) = e^{-i \lambda_i t} u_i(x)$, with $\lambda_i \in \R$ and $u_i:\R^N \to \R$. This ansatz leads to the following elliptic system for the densities $u_1$ and $u_2$:
\begin{subequations}\label{complete problem}
\begin{equation}\label{system}
\begin{cases}
-\Delta u_1- \lambda_1 u_1 = \mu_1 u_1^3+ \beta  u_2^{2} u_1 \\
-\Delta u_2- \lambda_2 u_2 = \mu_2 u_2^3 +\beta u_1^{2} u_2
\end{cases} \text{in $\R^N$}.
\end{equation}
This paper concerns the existence of \emph{normalized solutions} to \eqref{system} in dimension $N=3$, i.e. the existence of real numbers $(\lambda_1,\lambda_2) \in \R^2$ and of functions $(u_1,u_2) \in H^1(\R^3,\R^2)$ satisfying \eqref{system} together with the normalization condition
\begin{equation}\label{normalization}
\int_{\R^3} u_1^2 = a_1^2 \quad \text{and} \quad \int_{\R^3} u_2^2 = a_2^2,
\end{equation}
\end{subequations}
for a-priori given $a_1$, $a_2>0$, $\mu_1,\mu_2$, $\beta \in \R$. In what follows we refer to a solution of \eqref{system}-\eqref{normalization} simply as to a solution to \eqref{complete problem}.
We emphasize that, prescribing the \emph{masses} $a_i$ from the beginning, the \emph{frequencies} $\lambda_i$ are included in the unknown.
A somehow dual approach consists in fixing the frequencies $\lambda_i$ from the beginning, and leave the masses free.

Normalized solutions are particularly interesting from a physical point of view, since the mass $\|\Phi_i(t,\cdot)\|_{L^2}= \|u_i\|_{L^2}$ has often a clear physical meaning. In the aforementioned contexts, it represents the number of particles of each component in Bose-Einstein condensates, or the power supply in the nonlinear optics framework.
But despite this physical relevance, most of the papers deal with the problem with fixed frequencies, see e.g. \cite{AmCo, BaWa, ChZo, LinWei, LiuWang, MaiaMontefuscoPellacci, Mand, SaWa, Sir, So, SoTa, TerVer, WeiWeth} and the references therein, while problem \eqref{complete problem} is far from being well understood.

In order to clarify the difficulties that one has to face when searching for normalized solutions, in what follows we introduce some notation and review the few known results regarding \eqref{complete problem}.

Let $\mu_i=: \beta_{ii}$, $\beta=: \beta_{12}=\beta_{21}$, and for any $a>0$ let us consider
\begin{equation}\label{def sphere}
S_a:=\left\{ u \in H^1(\R^3): \int_{\R^3} u^2 = a^2 \right\}.
\end{equation}
Solutions to \eqref{complete problem} are critical points of the \emph{energy functional}
\begin{equation}\label{def energy}
J(u_1,u_2) = \int_{\R^3} \frac12 \sum_{i=1}^2 |\nabla u_i|^2 - \frac14 \sum_{i,j=1}^2 \beta_{ij}  u_i^2 u_j^2,
\end{equation}
on the constraint $S_{a_1} \times S_{a_2}$ with $(\lambda_1,\lambda_2)$ Lagrange multipliers. We are interested in \emph{positive} normalized solutions, i.e.\ normalized solutions with $u_1, u_2 >0$ in $\R^N$. Concerning the terminology, we often identify a solution $(\lambda_1,\lambda_2,u_1,u_2)$ of \eqref{complete problem} with its last components $(u_1,u_2)$, with some abuse of notation. This is justified by the fact that we obtain $(u_1,u_2)$ as critical points of the above constrained functional and $(\lambda_1,\lambda_2)$ are determined as Lagrange multipliers.

Some papers concern the existence of positive normalized solution when $\R^N$ is replaced by a bounded domain $\Omega$, or when a trapping potential is included in the equation; we refer to \cite{NoTaVe}, where essentially no assumption is imposed on $\mu_1$, $\mu_2$, $\beta$, but where the masses $a_1$ and $a_2$ are supposed to be small, and to \cite{NoTaTeVe, TaTe}, which regard the \emph{defocusing-repulsing} case $\mu_1,\mu_2,\beta<0$ with equal masses $a_1=a_2=1$. Notice that, if $\mu_1,\mu_2,\beta<0$ and $\Omega$ is bounded, the existence of a single normalized solution can be proved quite easily by minimization arguments, and indeed in \cite{NoTaTeVe, TaTe} the authors are mainly interested in multiplicity results and occurrence of phase-separation.

Let us consider now the \emph{focusing} case $\mu_1,\mu_2>0$ in the whole space $\R^N$. When \eqref{complete problem} is considered in dimension $N=1$, the constrained functional is bounded from below, and for arbitrary $a_i, \mu_i,\beta>0$ a positive normalized solution can be found minimizing $J|_{S_{a_1} \times S_{a_2}}$ and using concentration-compactness arguments. This approach, used in \cite{NguWan} (see also \cite[Section 5]{Caoetal}), fails if $N=2,3$, since $J|_{S_{a_1} \times S_{a_2}}$ is unbounded both from above and from below.
Thus, in higher dimensions one is induced to apply minimax methods, as successfully done in \cite{BaJeSo}. In the paper \cite{BaJeSo} we considered the \emph{attractive} case $\beta>0$ in $\R^3$ (the $2$-dimensional case is particularly delicate, see the forthcoming Remark \ref{rem: L^2 critical}). We proved that, for arbitrary masses $a_i$ and parameters $\mu_i$, there exist $\bar \beta_2> \bar \beta_1>0$ (depending on the data) such that for both $0 < \beta <\bar \beta_1$ and $\beta>\bar \beta_2$ system \eqref{complete problem} has a positive radial solution; in case $\beta>\bar \beta_2$ this solution is of mountain pass type, while for $0<\beta<\bar \beta_1$ the solution is obtained with a $2$-dimensional linking. This is somehow reminiscent to what happens for the unconstrained problem with fixed frequencies \cite{AmCo, Sir}. But despite the similarity between the results in \cite{BaJeSo} and those in \cite{AmCo, Sir}, the proofs differ substantially: the approach in \cite{AmCo, Sir} is indeed based on the research of critical points for the \emph{action functional}
\[
\mathcal{A}(u_1,u_2) := J(u_1,u_2) - \sum_{i=1}^2 \frac{\lambda_i}{2} \int_{\R^3} u_i^2
\]
constrained on Nehari-type sets associated to the problem, while apparently no Nehari manifold is available in the framework of normalized solutions because $\lambda_1$ and $\lambda_2$ are part of the unknown, and $(u_1,u_2)$ cannot be used as variation for $J|_{S_{a_1} \times S_{a_2}}$ in $(u_1,u_2)$. Further difficulties in dealing with the normalization constraint are that the existence of \emph{bounded} Palais-Smale sequences requires new arguments (the classical method used to prove the boundedness of any Palais-Smale sequence for unconstrained Sobolev-subcritical problem does not work), that Lagrange multipliers have to be controlled, and that weak limits of Palais-Smale sequences do not necessarily lie on $S_{a_1} \times S_{a_2}$. For all these reasons, the proofs in \cite{BaJeSo} are quite delicate and cannot be directly extended to cover the case $\beta<0$. The existence of normalized solutions for the focusing-repulsive case $\mu_i>0$, $\beta<0$ was then, up to now, completely open. This is the object of our first main result.


\begin{theorem}\label{thm: beta<0}
Let $N=3$, and let $\mu_1,\mu_2,a_1,a_2>0$ and $\beta<0$ be fixed. Then \eqref{complete problem} has a solution $(\lambda_1,\lambda_2,\bar u_1,\bar u_2)$ with $\lambda_i<0$, and $\bar u_i$ is positive in $\R^3$ and radially symmetric.
 \end{theorem}

For $a_1$, $a_2$, $\mu_1$ and $\mu_2$ fixed, we find then a family $\{(\lambda_{1,\beta},\lambda_{2,\beta},\bar u_{1,\beta},\bar u_{2,\beta}): \beta<0\}$. Our next result shows that phase-separation occurs as $\beta\to-\infty$.

\begin{theorem}\label{thm: phase sep}
Let $N=3$, and let $\mu_1,\mu_2,a_1,a_2>0$ be fixed. Then, as $\beta \to -\infty$, we have (up to a subsequence):
\begin{itemize}
\item[($i$)] $(\lambda_{1,\beta},\lambda_{2,\beta}) \to (\lambda_1,\lambda_2)$, with $\lambda_1,\lambda_2 \le 0$;
\item[($ii$)] $(\bar u_{1,\beta},\bar u_{2,\beta}) \to (\bar u_1,\bar u_2)$ in $\mathcal{C}^{0,\alpha}_{\loc}(\R^N)$ and in $H^1_{\loc}(\R^N)$;
\item[($iii$)] $\bar u_{1}$ and $\bar u_{2}$ are nonnegative Lipschitz continuous functions having disjoint positivity sets, in the sense that $\bar u_1 \bar u_2 \equiv 0$ in $\R^N$;
\item[($iv$)] the difference $\bar u_1-\bar u_2$ is a sign-changing radial solution of
\[
-\Delta w -\lambda_1 w^+ +\lambda_2 w^- = \mu_1 (w_1^+)^3 - \mu_2(w_1^-)^3 \qquad \text{in $\R^N$}.
\]
\end{itemize}
\end{theorem}

In order to prove Theorem \ref{thm: beta<0}, we devise a new approach, substantially different with respect to the one in \cite{BaJeSo}, based upon the introduction of a further constraint. Let
\begin{equation}\label{def G}
G(u_1,u_2) = \sum_{i=1}^2 \int_{\R^3} |\nabla u_i|^2 - \frac{3}{4}\sum_{i,j=1}^2  \int_{\R^3} \beta_{ij} u_i^2 u_j^2,
\end{equation}
and let
\begin{equation}\label{definition Pohozaev}
\mathcal{P}:= \left\{ (u_1,u_2) \in S_{a_1} \times S_{a_2} \mid G(u_1,u_2)=0 \right\}.
\end{equation}
As proved in \cite[Lemma 4.6]{BaJeSo}, any solution of \eqref{complete problem} stays in $\mathcal{P}$, the equation $G(u_1,u_2)=0$ being the Pohozaev identity for \eqref{complete problem}. The solution that was obtained in \cite[Theorem~1.2]{BaJeSo} for $\beta>0$ large by a mountain pass argument on $S_{a_1} \times S_{a_2}$ was characterized as minimizer of $J$ on $\mathcal{P}$. In the present paper we show that one can actually apply min-max methods to $J$ constrained to $\mathcal{P}$ in order to obtain solutions of \eqref{complete problem}.

\begin{theorem}\label{thm: natural intro}
The set $\mathcal{P}$ is a $\mathcal{C}^1$-manifold of codimension $1$ in $S_{a_1} \times S_{a_2}$, and moreover:
\begin{itemize}
\item[($i$)] If there exists a Palais-Smale sequence $\{(\tilde u_{1,n}, \tilde u_{2,n})\}$ for $J$ restricted to $\mathcal{P}$ at level $\ell \in \R$, then there exists a possibly different Palais-Smale sequence $\{(u_{1,n}, u_{2,n})\} \subset \mathcal{C}^\infty_c(\R^3)$ for $J$ restricted to $S_{a_1} \times S_{a_2}$ at the same level $\ell \in \R$.
\item[($ii$)] If $(u_1,u_2)$ is a critical point of $J$ restricted on $\mathcal{P}$, then $(u_1,u_2)$ is a critical point of $J$ restricted on $S_{a_1} \times S_{a_2}$, and hence a solution to \eqref{complete problem}.
\end{itemize}
\end{theorem}

One often refers to property ($ii$) saying that \emph{$\mathcal{P}$ is a natural constraint}. As far as we know this is the first example of a natural constraint when dealing with normalized solutions; see the Remark \ref{rem: nat const} below for a more extended discussion.

With the previous result in hands, we prove Theorem \ref{thm: beta<0} finding a critical point of mountain pass type for the constrained functional $J|_{\mathcal{P}}$.


We point out that this natural constraint approach is very flexible and, suitably modified, permits also to recover the known existence and multiplicity results regarding normalized solutions for the nonlinear Schr\"odinger equation
\begin{equation}\label{single}
\begin{cases}
-\Delta u -\lambda u = f(u) & \text{in $\R^N$} \\
u >0, u \in H^1(\R^N) \\
\int_{\R^N} u^2 = a^2,
\end{cases}
\end{equation}
under appropriate assumptions on $f$. Solutions to \eqref{single} are critical points of the functional
\begin{equation}\label{functional single intro}
I(u):= \int_{\R^N} |\nabla u|^2 - F(u), \quad F(s):= \int_0^s f(\sigma)\,d\sigma,
\end{equation}
on the sphere $S_a$. The case of the pure power nonlinearity $f(s) = |s|^{p-2} s$ can be treated using the results available for the problem with fixed $\lambda<0$, properly scaling the equation; such an approach fails when $f$ is inhomogeneous. For inhomogeneous $f$ two different pictures are possible, depending on whether or not $I$ can be globally minimized on $S_{a}$. For the power nonlinearity, the former case, called \emph{$L^2$-subcritical}, takes place if $2<p<2+4/N$, and was firstly considered in \cite{Stu1,Stu2}. Afterwards it was also addressed with the aid of the concentration-compactness principle \cite{Lio1,Lio2}. If $2+4/N \le  p <2N/(N-2)$, then $I|_{S_a}$ cannot be minimized, and the problem is considerably more involved. The so called \emph{$L^2$-critical} case $p=2+4/N$ is particularly delicate, and will be discussed in Remark \ref{rem: L^2 critical}. The $L^2$-supercritical and Sobolev subcritical case $2+4/N <  p <2N/(N-2)$ was considered only in the two papers \cite{Jea,BaDeV}. In \cite{Jea} it is proved the existence of a mountain pass positive normalized solution. In \cite{BaDeV}, putting in evidence the ``fountain" type structure of $I|_{S_a}$, the authors proved the existence of infinitely many normalized solutions. The precise assumptions considered in \cite{Jea,BaDeV} on the nonlinearity $f$ are the following:
\begin{itemize}
\item[($f1$)] $f: \R \to \R$ is continuous and odd;
\item[($f2$)] $N \ge 2$, and there exists $\alpha, \beta \in \R$,
\[
2+\frac{4}{N} < \alpha \le \beta < 2^*:= \begin{cases} +\infty & \text{if $N=1,2$} \\ \frac{2N}{N-2} \end{cases},
\]
such that
\[
0 < \alpha F(s) \le f(s) s \le \beta F(s) \qquad \forall s \in \R \setminus \{0\};
\]
\end{itemize}

In this paper we give an alternative simple proof of the existence and multiplicity results in \cite{BaDeV, Jea}. We emphasize that here we use the additional assumption ($f3$) below, which is not needed in \cite{BaDeV, Jea}.

\begin{theorem}\label{thm: main single}
Let $N \ge 2$, $a>0$, and let $f$ satisfy ($f1$), ($f2$), and
\begin{itemize}
\item[($f3$)] the map $\tilde F(s):= f(s)s - 2F(s)$ is of class $\C^1$, and
\[
\tilde F'(s) s > \left(2+\frac{4}{N}\right) \tilde F(s).
\]
\end{itemize}
Then \eqref{single} has infinitely many radial solutions $\{u_k: k \ge 1\}$ with increasing energy, and $u_1$ is positive in $\R^N$.
\end{theorem}

Our proof of Theorem \ref{thm: main single} is based upon the search for critical points of $I$ constrained on
\begin{equation}\label{def M intro}
\mathcal{M}  :=\left\{ u \in S_a: G(u) = 0\right\},
\end{equation}
where
\begin{equation}\label{def G single}
\begin{split}
G(u) :&=\int_{\R^N} |\nabla u|^2 - \int_{\R^N}  \left( \frac{N}{2} f(u)u-N F(u) \right) \\
& = \int_{\R^N} |\nabla u|^2 - \frac{N}{2} \int_{\R^N} \tilde F(u).
\end{split}
\end{equation}
%
It turns out that $\mathcal{M}$ is a natural constraint, as expressed by the following statement.

\begin{theorem}\label{thm: constraint single intro}
Under ($f1$)-($f3$), the set $\mathcal{M}$ is a $\C^1$ manifold, and moreover:
\begin{itemize}
%
\item[($i$)] If there exists a Palais-Smale sequence $\{\tilde u_{n}\}$ for $I$ restricted to $\mathcal{M}$ at level $\ell \in \R$, then there exists a possibly different Palais-Smale sequence $\{u_{n}\} \subset \mathcal{C}^\infty_c(\R^3)$ for $I$ restricted to $S_a$ at the same level $\ell \in \R$.
\item[($ii$)] If $u$ is a critical point of $I$ restricted on $\mathcal{M}$, then $u$ is a critical point of $I$ restricted on $S_{a}$, and hence a solution to \eqref{single}.
\end{itemize}
\end{theorem}

We will see that the constrained functional $I$ restricted to $\mathcal{M}$ is bounded from below, coercive, and satisfies the Palais-Smale condition. Therefore, Theorem \ref{thm: main single} will be a simple consequence of the equivariant Lusternik-Schirelman theory.

\begin{remark}
Assumption ($f3$) is not needed in \cite{BaDeV, Jea} for proving the existence of solutions of \eqref{single} (actually it is required in \cite{Jea} in order to treat the case $N=1$). It is an interesting question whether ($f3$) can be omitted in Theorem~\ref{thm: main single} also with our approach. Then $\cM$ will not be a manifold anymore but it still contains all solutions of \eqref{single}. This suggests that Theorem~\ref{thm: main single} could be approached using the critical point theory on metric spaces from \cite{CDM}. In any case, we observe that for a wide class of nonlinearities, such as those of type
\[
f(s) = \sum_{i=1}^m \mu_i |s|^{p_i-2} s \quad \text{with} \quad \mu_i >0,
\]
this assumption is already included in ($f2$). Notice also that, even if $\tilde F \in \C^1$, the function $f$ need not be differentiable in the origin.
\end{remark}

\begin{remark}\label{rem: nat const}
Roughly speaking, the manifolds $\mathcal{P},\cM$ play, for problems \eqref{complete problem}, \eqref{single}, respectively, the role that the Nehari manifold plays for equations or systems with fixed frequencies (see e.g.\ \cite{AmbMal}, and the references therein). Recall that the Nehari manifold is defined by the equation $d\mathcal{A}(u)u=0$ ($\mathcal{A}$ denotes the action functional) which is not available when looking for normalized solutions. For scalar equations without normalization constraint the Pohozaev manifold has been used in \cite{Sha} in order to find a ground state solution of a nonlinear Klein-Gordon equation. More recently the Pohozaev manifold has been used in \cite{AzzPom,JeaTan,LehMai} for investigating nonlinear (scalar) Schr\"odinger equations (without normalization constraint). In \cite{AzzPom,JeaTan,LehMai,Sha} the nonlinearity was such that the equation $d\mathcal{A}(u)u=0$ did not necessarily define a manifold whence the authors worked with the Pohozaev identity instead. The authors of \cite{AzzPom,JeaTan,Sha} were interested in least-energy solutions. They showed that the mountain pass solution corresponds to a minimizer of the associated functional on the Pohozaev manifold. In \cite{LehMai} the authors set up a min-max scheme for the functional constrained to $\cP$, obtaining a Palais-Smale (or Cerami) sequence for the constrained functional. (It is unclear how to obtain a Cerami sequence in the full space as claimed in \cite{LehMai}. Probably a result like Theorem~\ref{thm: natural intro} or Theorem~\ref{thm: constraint single intro} is needed.) Observe that in \cite{LehMai} the authors use the derivative $d\mathcal{A}(u)u$ in an essential way; this is not available in the normalized setting.
It seems that our paper is the first to set up critical point theory on the Pohozaev manifold under the $L^2$ constraint.
\end{remark}

%

We conclude the introduction mentioning further problems which we believe could be treated with our natural constraint approach, and discussing why we do not consider \eqref{complete problem} in $\R^2$.

\begin{remark}
Even though we focused on system \eqref{system}, we can treat more general power type problems such as
 \begin{equation}\label{power sys}
\begin{cases}
-\Delta u_1 -\lambda_1 u_1 = \mu_1 |u_1|^{p_1-2} u_1 +  \beta |u_1|^{r_1-2} |u_2|^{r_2} u_1 & \text{in $\R^N$}\\
-\Delta u_2 -\lambda_2 u_2 = \mu_2 |u_2|^{p_2-2} u_2 + \beta |u_1|^{r_1} |u_2|^{r_2-2} u_2 & \text{in $\R^N$}\\
\int_{\R^N} u_1^2 = a_1^2 \quad \int_{\R^N} u_2^2 = a_2^2,
\end{cases}
\end{equation}
or even systems with right hand sides $\pa_1 F(u_1,u_2)$, $\pa_2 F(u_1,u_2)$, under appropriate assumptions on $F$. Systems with an arbitrary number of components can be considered as well, i.e.\ also in these contexts it is possible to introduce the set $\mathcal{P}$, and to prove that it is a natural constraint, in the sense specified by Theorem \ref{thm: natural intro}. Notice that \eqref{complete problem} is a particular case of \eqref{power sys}, and we mention that existence results under different assumptions on the data of the problem have been obtained in \cite{BarJea,BaJeSo,GuoJea}. We believe that some of the results therein could be re-proved using $\mathcal{P}$ and adapting the method used here.

More generally, we believe that our approach can be adapted in many situations when we search for normalized solutions and a Pohozaev-type identity is available. In any case, several complications could arise.

With regard to this, we mention that the three problems \eqref{complete problem}, \eqref{single} and \eqref{power sys} considered in this paper have been studied also in bounded domains instead than in the whole space, see \cite{FibMer, NoTaVe2, NoTaVe} and the references therein. In such situations it is not clear how to define $\mathcal{P}$ or $\mathcal{M}$, since the Pohozaev identity involves boundary terms which are not necessarily well defined for $u \in H^1(\Omega)$.

Also the case when a non-autonomous potential is added in the equation (as done, for instance, in \cite{BufEstSer}, where a strongly indefinite problem is considered) deserves some special care. Indeed, our technique relies on the possibility of scaling/dilating the equation, and the presence of a potential would require extra efforts to be treated.
\end{remark}

\begin{remark}\label{rem: L^2 critical}
The existence of normalized solutions in the $L^2$-critical case is a very delicate problem. Let us consider the stationary NLS equation
\begin{equation}\label{pb rem}
-\Delta u -\lambda u = |u|^{p-2}u  \quad \text{in $\R^N$},  \quad \int_{\R^N} u^2 = a^2.
\end{equation}
If either $2<p<2+4/N$ or $2+4/N <p <2N/(N-2)$, for any $a>0$ the problem has a unique positive radial solution, which can be obtained by scaling the unique positive radial solution of
\[
-\Delta w+ w= |w|^{p-2} w \quad \text{in $\R^N$}.
\]
If on the other hand $p=2+4/N$, which is for instance the case of the cubic NLS equation (i.e.\ $p=4$) in $\R^2$, then there exists a uniquely determined $\bar a>0$ (depending only on the dimension) such that \eqref{pb rem} with $a=\bar a$ has infinitely many positive radial solutions (corresponding to different $\lambda$), while for $a \neq \bar a$ \eqref{pb rem} has no positive solution at all. This makes the $L^2$-critical problem extremely peculiar to treat, and as far as we know there is no result concerning inhomogeneous $f$ in this case. In the same spirit, even though we could introduce the set $\mathcal{P}$, we cannot treat system \eqref{system} in $\R^2$ with our technique, which is tailor-made for the $L^2$-supercritical and Sobolev-subcritical context.
\end{remark}

\paragraph{\textbf{Organization of the paper}} Theorem \ref{thm: natural intro} is the object of Section \ref{sec: natural}. The result is then used in the proof of existence of solutions to \eqref{complete problem}, Theorem \ref{thm: beta<0}, which is the content of Section \ref{sec: existence}. Theorem \ref{thm: phase sep}, is treated in Subsection \ref{sec: phase sep}. Sections \ref{sec: natural single} and \ref{sec: existence single} are devoted to the proofs of Theorems \ref{thm: constraint single intro} and \ref{thm: main single} respectively.

\medskip

\paragraph{\textbf{Notation}} For the sake of brevity, we often write $\mf{u}$ instead of $(u_1,u_2)$ for vector valued functions in $H^1(\R^3,\R^2)$. We recall that $\beta_{ii}:= \mu_i$ and $\beta_{12}=\beta_{21}:=\beta$. If $\mathcal{N}$ is a $\mathcal{C}^1$-manifold, we denote by $T_P \mathcal{N}$ the tangent space to $\mathcal{N}$ in the point $P\in\mathcal{N}$. Throughout the paper $C$ will always denote a positive constant, whose value is allowed to change also from line to line.

\section{A natural constraint for elliptic systems}\label{sec: natural}

In this section we aim at proving that the set $\mathcal{P}$, introduced in \eqref{definition Pohozaev}, is a natural constraint in the sense specified by Theorem \ref{thm: natural intro}. Actually, we will prove the following slightly stronger statement.

\begin{theorem}\label{thm: natural}
The set $\mathcal{P}\subset S_{a_1} \times S_{a_2}\subset H^1(\R^3,\R^2)$ is a $\mathcal{C}^1$-submanifold, and moreover:
\begin{itemize}
\item[($i$)] If $\{(u_{1,n},u_{2,n})\} \subset \mathcal{C}^\infty_c(\R^3) \cap \mathcal{P}$ is a Palais-Smale sequence for $J$ restricted to $\mathcal{P}$ at a certain level $\ell \in \R$, then  $\{(u_{1,n},u_{2,n})\}$ is a Palais-Smale sequence for $J$ restricted to $S_{a_1} \times S_{a_2}$.

\item[($ii$)] If there exists a Palais-Smale sequence $\{(\tilde u_{1,n},\tilde u_{2,n})\}$ for $J$ restricted to $\mathcal{P}$ at level $\ell \in \R$, then there exists a possibly different Palais-Smale sequence $\{( u_{1,n},  u_{2,n})\} \subset \mathcal{C}^\infty_c(\R^3)$ for $J$ restricted to $\mathcal{P}$ at the same level $\ell \in \R$. Moreover $\|u_{i,n}-\tilde u_{i,n}\|_{H^1} \to 0$ as $n \to \infty$ for $i=1,2$.

\item[($iii$)] If there exists a Palais-Smale sequence $\{(\tilde u_{1,n}, \tilde u_{2,n})\}$ for $J$ restricted to $\mathcal{P}$ at level $\ell \in \R$, then there exists a possibly different Palais-Smale sequence $\{(u_{1,n}, u_{2,n})\} \subset \mathcal{C}^\infty_c(\R^3)$ for $J$ restricted to $S_{a_1} \times S_{a_2}$ at the same level $\ell \in \R$. Moreover $\|u_{i,n}-\tilde u_{i,n}\|_{H^1} \to 0$ as $n \to \infty$ for $i=1,2$.

\item[($iv$)] Let $(u_1,u_2)$ be a critical point of $J$ restricted on $\mathcal{P}$. Then $(u_1,u_2)$ is a critical point of $J$ restricted on $S_{a_1} \times S_{a_2}$, and hence a solution to \eqref{complete problem}.
\end{itemize}
\end{theorem}


%

The first step consists in showing that $\mathcal{P}$ is a manifold.

\begin{lemma}\label{lem: manifold}
The set $\mathcal{P}$ is a $\mathcal{C}^1$-submanifold of codimension $1$ in $S_{a_1} \times S_{a_2}$, hence a $\mathcal{C}^1$-submanifold of codimension $3$ in $H^1(\R^3,\R^2)$.
\end{lemma}

\begin{proof}
As subset of $H^1(\R^3,\R^2)$, the constraint $\mathcal{P}$ is defined by $G(u_1,u_2)=0$, $G_1(u_1)=0$, $G_2(u_2)=0$, where
\[
G_i(u_i):= a_i^2 - \int_{\R^3} u_i^2
\]
for $i=1,2$, and $G$ is defined in \eqref{def G}. Since the functions $G$ and $G_i$ are of class $\mathcal{C}^1$, we have only to check that
\begin{equation}\label{claim surjectivity}
d(G_1, G_2, G): H^1(\R^3,\R^2) \to \R^3 \quad \text{is surjective}.
\end{equation}
If this is not true, $dG(u_1,u_2)$ has to be linearly dependent from $dG_1(u_1)$ and $dG_2(u_2)$, i.e. there exist $\nu_1,\nu_2 \in \R$ such that
\[
2\sum_{i=1}^2 \int_{\R^3} \nabla u_i  \cdot \nabla \varphi_i - \frac{3}{2}\sum_{i,j=1}^2 \int_{\R^3} \beta_{ij} u_i u_j (u_i \varphi_j + \varphi_i u_j) = 2\sum_{i=1}^2 \nu_i \int_{\R^3} u_i \varphi_i
\]
for every $(\varphi_1,\varphi_2) \in H^1(\R^3,\R^2)$. This means that $(u_1,u_2)$ is a solution to
\[
\begin{cases}
-\Delta u_i - \nu_i u_i = \frac{3}{2}\sum_{i=1}^2 \beta_{ij} u_i u_j^2 \\
\int_{\R^3} u_i^2 = a_i^2
\end{cases}  \text{in $\R^3$},
\]
for $i=1,2$. But then, applying \cite[Lemma 4.6]{BaJeSo}, we conclude that
\[
\sum_{i=1}^2 \int_{\R^3} |\nabla u_i|^2 - \frac{9}{8} \sum_{i,j=1}^2\int_{\R^3}\beta_{ij} u_i^2 u_j^2 = 0.
\]
Recalling that $G(u_1,u_2) = 0$, this implies that
\[
\sum_{i=1}^2 \int_{\R^3} |\nabla u_i|^2 = 0, \quad \text{and hence} \quad (u_1,u_2) = (0,0),
\]
in contradiction with the fact that $(u_1,u_2) \in S_{a_1} \times S_{a_2}$.
\end{proof}

We define for $s \in \R$ and $w \in H^1(\R^3)$ the function
\begin{equation}\label{def star}
(s \star w)(x) = e^{3s/2} w(e^s x).
\end{equation}
One can easily check that $\|s \star w\|_{L^2(\R^3)} = \|w\|_{L^2(\R^3)}$ for every $s \in \R$. As a consequence, given $(u_1,u_2) \in S_{a_1} \times S_{a_2}$, it results that
\[
s \star \mf{u} = s \star (u_1,u_2)  := (s \star u_1, s \star u_2) \in S_{a_1} \times S_{a_2}
\]
for every $s \in \R$. We consider the real valued function
\[
\Psi_{\mf{u}}(s):= J(s \star \mf{u}).
\]
By changing variables in the integrals, we obtain
\begin{equation}\label{explicit expression}
\Psi_{\mf{u}}(s)  = \frac{e^{2s}}{2}\int_{\R^3} \sum_{i=1}^2 |\nabla u_i|^2 - \frac{e^{3s}}{4}\int_{\R^3} \sum_{i,j=1}^2 \beta_{ij} u_i^2 u_j^2.
\end{equation}
Let us introduce
\begin{equation}\label{def cone}
\mathcal{E}:= \left\{(u_1,u_2) \in S_{a_1} \times S_{a_2}: \sum_{i,j=1}^2 \beta_{ij} \int_{\R^3}u_i^2 u_j^2 > 0\right\}.
\end{equation}
By the H\"older inequality, it follows straightforwardly that $\mathcal{E} = S_{a_1} \times S_{a_2}$ in case $- \sqrt{\mu_1 \mu_2} < \beta < +\infty$, while for $\beta \le - \sqrt{\mu_1 \mu_2}$ it results that $\mathcal{E}  \subset S_{a_1} \times S_{a_2}$ with strict inclusion. Notice also that, thanks to the continuity of the Sobolev embedding $H^1(\R^3) \hookrightarrow L^4(\R^3)$, the set $\mathcal{E}$ is an open subset of $S_{a_1} \times S_{a_2}$ in the $H^1$ topology. The role of $\mathcal{E}$ is clarified by the following statement.

\begin{lemma}\label{lem: structure P}
For any $\mf{u}=(u_1,u_2) \in S_{a_1} \times S_{a_2}$, a value $s \in \R$ is a critical point of $\Psi_{\mf{u}}$ if and only if $s \star \mf{u} \in \mathcal{P}$. It results that:
\begin{itemize}
\item[($i$)] If $\mf{u} \in \mathcal{E}$, then there exists a unique critical point $s_{\mf{u}} \in \R$ for $\Psi_{\mf{u}}$, which is a strict maximum point, and is defined by
\begin{equation}\label{def s_u}
\exp(s_{\mf{u}}) = \frac{4 \int_{\R^3} \sum_i |\nabla u_i|^2}{3 \int_{\R^3} \sum_{i,j} \beta_{ij} u_i^2 u_j^2 }.
\end{equation}
In particular, if $\mf{u} \in \mathcal{P}$, then $s_{\mf{u}}=0$.
\item[($ii$)] If $\mf{u}=(u_1,u_2) \not \in \mathcal{E}$, then $\Psi_{\mf{u}}$ has no critical points in $\R$.
\end{itemize}
\end{lemma}
The proof is a simple consequence of \eqref{explicit expression} and the definition of $\mathcal{P}$ and $\mathcal{E}$.

\medskip

In the following statement we describe the structure of $T_{\mf{u}} (S_{a_1} \times S_{a_2})$ in points of $\mathcal{P}$.

\begin{lemma}\label{lem: splitting}
For any $\mf{u} \in \mathcal{P} \cap \mathcal{C}^\infty_c(\R^3,\R^2)$, we have
\[
T_{\mf{u}} (S_{a_1} \times S_{a_2}) = T_{\mf{u}} S_{a_1} \times T_{\mf{u}} S_{a_2}
 = T_{\mf{u}} \mathcal{P} \oplus \R \left. \frac{d}{ds}\right|_{s=0}(s \star \mf{u}).
\]
\end{lemma}

\begin{proof}
First observe that for $w\in S_a\cap \mathcal{C}^\infty_c(\R^3,\R)$ the path $\gamma:\R\to S_a$, $s\mapsto s*w$, is of class $C^1$ with derivative given by $\gamma'(s)(x)=\frac{3}{2}e^{3s/2} w(e^s x) + e^{3s/2} \nabla w(e^s x) \cdot (e^s x)$. Consequently $\left.\frac{d}{ds}\right|_{s=0}(s \star \mf{u})\in T_{\mf{u}}$ is well defined for $\mf{u} \in \mathcal{P} \cap \mathcal{C}^\infty_c(\R^3,\R^2)$. By Lemma \ref{lem: manifold}, we know that $\mathcal{P}$ has codimension $1$ with respect to $S_{a_1} \times S_{a_2}$, and hence it is sufficient to show that
\[
 \left. \frac{d}{ds}\right|_{s=0} (s \star \mf{u}) \notin T_{\mf{u}} \mathcal{P},
\]
that is
\[
 dG(u_1,u_2) \left[\left. \frac{d}{ds}\right|_{s=0}(s \star \mf{u})\right] \neq 0,
\]
with $G$ defined in \eqref{def G}. For any $w \in \mathcal{C}^\infty_c(\R^3)$, we can compute
\begin{equation}\label{eq: dG on variation}
\begin{split}
dG(u_1,u_2) & \left[\left. \frac{d}{ds}\right|_{s=0}(s \star \mf{u})\right] = 2\sum_i \int_{\R^3} \left[\frac{3}{2} |\nabla u_i|^2
 + \nabla u_i \cdot \nabla (\nabla u_i \cdot x) \right] \\
 & -\frac{9}{2} \sum_{i,j}\int_{\R^3}  \beta_{ij} u_i^2 u_j^2
-3 \sum_{i,j} \int_{\R^3} \beta_{ij} u_i u_j^2 \nabla u_i \cdot x
\end{split}
\end{equation}
Observing that
\[
\nabla(\nabla u_i \cdot x) = (\nabla^2 u_i)[x] + \nabla u_i
\]
and using the divergence theorem, the first integral on the right hand side in \eqref{eq: dG on variation} can be developed as
\begin{multline*}
\int_{\R^3}   \left[\frac{3}{2} |\nabla u_i|^2
 + \nabla u_i \cdot \nabla (\nabla u_i \cdot x) \right]  \\
=  \int_{\R^3} \left[\frac{3}{2} |\nabla u_i|^2
 + \frac{1}{2}\nabla(|\nabla u_i|^2) \cdot x + |\nabla u_i|^2\right]
 = \int_{\R^3} |\nabla u_i|^2.
 \end{multline*}
Concerning the second and the third integral, again by the divergence theorem it results that
\begin{multline*}
\frac{9}{2} \sum_{i,j}\int_{\R^3}  \beta_{ij} u_i^2 u_j^2
+3 \sum_{i,j} \int_{\R^3} \beta_{ij} u_i u_j^2 \nabla u_i \cdot x  \\ = \frac{9}{2} \sum_{i,j}\int_{\R^3}  \beta_{ij} u_i^2 u_j^2 + \frac{3}{4} \sum_{i,j}\int_{\R^3}  \beta_{ij} \nabla(u_i^2 u_j^2)\cdot x
=  \frac{9}{4} \sum_{i,j}\int_{\R^3}  \beta_{ij} u_i^2 u_j^2.
\end{multline*}
Coming back to \eqref{eq: dG on variation}, and using the definition of $\mathcal{P}$, we finally conclude

\begin{multline*}
dG(u_1,u_2) \left[\left. \frac{d}{ds}(s \star (u_1,u_2))\right|_{s=0}\right] \\
= 2 \sum_i \int_{\R^3} |\nabla u_i|^2 - \frac{9}{4} \sum_{i,j}  \int_{\R^3}  \beta_{ij} u_i^2 u_j^2
= - \sum_i \int_{\R^3} |\nabla u_i|^2 \neq 0,
\end{multline*}
which completes the proof.
\end{proof}

\begin{remark}
In general the variation
\[
 \left. \frac{d}{ds}\right|_{s=0}(s \star \mf{u})
\]
is not in $H^1(\R^3,\R^2)$; this is why we require $\mf{u} \in \mathcal{C}^\infty_c(\R^3,\R^2)$ in the lemma. Actually it would have been enough to suppose that $\mf{u}\in H^2(\R^3,\R^2)$ decays sufficiently fast so that the previous variation stays in $H^1(\R^3,\R^2)$.
\end{remark}

In Lemma \ref{lem: splitting}, we showed that the tangent space to $S_{a_1} \times S_{a_2}$ in a point $\mf{u} \in \mathcal{P} \cap \mathcal{C}^\infty_c(\R^3,\R^2)$ splits as direct sum of the tangent space to $\mathcal{P}$ plus a $1$-dimensional subspace of type $\R(v_1,v_2)$ for a suitable variation $(v_1,v_2)$.
The crucial fact for Theorem \ref{thm: natural} is that any point of $\mathcal{P}$ is critical for $J$, by definition, with respect to variations in $\R(v_1,v_2)$. This is why criticality on $\mathcal{P}$ implies criticality on $S_{a_1} \times S_{a_2}$, which is rigorously proved in the following lemma.

\begin{lemma}\label{lem: crit P}
If $(u_1,u_2) \in \mathcal{C}^\infty_c(\R^3,\R^2) \cap \mathcal{P}$, then
\[
dJ(u_1,u_2) \left[\left. \frac{d}{ds}\right|_{s=0}(s \star (u_1,u_2))\right] = 0.
\]
\end{lemma}
\begin{proof}
If $\mf{u}=(u_1,u_2) \in \mathcal{C}^\infty_c(\R^3,\R^2) \cap \mathcal{P}$, then by Lemma \ref{lem: structure P} we have $s_{\mf{u}}=0$, and
\[
0 = \Psi_{\mf{u}}'(0) = \left. \frac{d}{ds}\right|_{s=0} J(s \star \mf{u})
 = \left[ dJ(s \star \mf{u})\left[ \frac{d}{ds}(s \star \mf{u})\right] \right]_{s=0}.
\]
The thesis follows.
\end{proof}

We prove now a simple preliminary result which we shall use many times in the rest of the paper.

\begin{lemma}\label{lem: s strong convergence}
Let $\{u_n\} \subset H^1(\R^3)$, $\{s_n\} \subset \R$, and let us suppose that $u_n \to u$ strongly in $H^1(\R^3)$ and $s_n \to s \in \R$, as $n \to \infty$. Then $s_n \star u_n \to s \star u$ strongly in $H^1(\R^3)$ as $n \to \infty$.
\end{lemma}
\begin{proof}
By definition $s_n \star u_n \to s \star u$ a. e. in $\R^3$, and
\[
\|s_n \star u_n\|_{H^1}^2 = e^{2s_n} \int_{\R^3} |\nabla u_n|^2 + \int_{\R^3} u_n^2  \le C
\]
for every $n$. Hence, up to a subsequence $s_n \star u_n \wc s \star u$ weakly in $H^1$, and moreover we have the convergence of the norms $\|s_n \star u_n\|_{H^1} \to \|s \star u\|_{H^1}$. This argument works for all the possible subsequences.
\end{proof}

We are finally ready for the:
\begin{proof}[Proof of Theorem \ref{thm: natural}]
For the proof of ($i$), let $\{\mf{u}_n\} \subset \mathcal{C}^\infty_c(\R^3,\R^2) \cap \mathcal{P}$ be a Palais-Smale sequence for $J|_{\mathcal{P}}$. We denote by $T^*_{\mf{u}}(S_{a_1} \times S_{a_2})$ the dual space to $T_{\mf{u}}(S_{a_1} \times S_{a_2})$, and by $\|\cdot\|$ the $H^1(\R^3,\R^2)$ norm. By Lemma \ref{lem: splitting}
\begin{align*}
\| dJ (\mf{u}_n) & \|_{T^*_{\mf{u}}(S_{a_1} \times S_{a_2})}
 = \sup\left\{\big|dJ(\mf{u}_n)[\mf{\varphi}]\big|: \mf{\varphi} \in T_{\mf{u}}(S_{a_1} \times S_{a_2}),\ \|\mf{\varphi}\| \le 1 \right\} \\
& = \sup\left\{ |dJ(\mf{u}_n)[\mf{\phi}] + dJ(\mf{u}_n)[\mf{\psi}] | :
     \begin{array}{l}\mf{\varphi} = \mf{\phi} + \mf{\psi}, \|\mf{\varphi}\| \le 1 \\
       \mf{\phi} \in T_{\mf{u}}\mathcal{P},  \ \mf{\psi} \in \R\left(\left.\frac{d}{ds}\right|_{s=0} (s \star \mf{u}_n)\right)
     \end{array}\right\}.
\end{align*}
Since Lemma \ref{lem: crit P} yields $dJ(\mf{u}_n)[\mf{\psi}]=0$, we deduce that
\begin{align*}
\| dJ (\mf{u}_n) \|_{(T_{\mf{u}}(S_{a_1} \times S_{a_2}))^*}
 &= \sup\left\{|dJ(\mf{u}_n)[\mf{\phi}]| : \mf{\phi} \in T_{\mf{u}}\mathcal{P},\ \|\mf{\phi}\| \le 1 \right\} \\
 & = \| dJ (\mf{u}_n)  \|_{(T_{\mf{u}}(\mathcal{P}))^*} \to 0
\end{align*}
as $n \to \infty$. Here we used the fact that $\{\mf{u}_n\}$ is a Palais-Smale sequence for $J$ restricted to $\mathcal{P}$. This proves point ($i$).

\medskip

Concerning ($ii$), we show first that $\mathcal{P} \cap \mathcal{C}^\infty_c(\R^3,\R^2)$ is dense in $\mathcal{P}$.
Let $\mf{u} \in \mathcal{P}$. By density in $H^1$, there exists $\{\mf{u}_n\} \subset \C^\infty_c(\R^3,\R^2) \cap (S_{a_1} \times S_{a_2})$ such that $\mf{u}_n \to \mf{u}$ strongly in $H^1(\R^3,\R^2)$. The problem is that $\mf{u}_n \not \in \mathcal{P}$ in general, but this can be easily settled in the following way: first, since $\mf{u} \in \mathcal{P} \subset \mathcal{E}$ with $\mathcal{E}$ from \eqref{def cone}, and since $\mathcal{E}$ is open, $\mf{u}_n \in \mathcal{E}$ for sufficiently large $n$. Then we can consider the uniquely determined $s_n:= s_{\mf{u}_n}$, defined by \eqref{def s_u}. By strong convergence, it is immediate that $s_n \to 0$ as $n \to +\infty$, so that Lemma \ref{lem: s strong convergence} implies that $s_n \star \mf{u}_n \to \mf{u}$ strongly in $H^1(\R^3)$. Moreover, by definition $s_n \star \mf{u}_n \in \mathcal{P}$ for every $n$, and hence the proof of the density is complete.

Let now $\{\tilde{\mf{u}}_n\}$ be a Palais-Smale sequence for $J$ on $\mathcal{P}$, and let $\eps_m \to 0^+$ as $m \to \infty$. For every $n$ and $m$, by density there exists $\mf{u}_{n,m} \in \mathcal{P} \cap \C^\infty_c(\R^3)$ such that $\|\mf{u}_{n,m}-\tilde{\mf{u}}_n\|_{H^1} <\eps_m$, and it is clear that the diagonal sequence $\mf{u}_n := \mf{u}_{n,n}$ satisfies all the requirements in point ($ii$).

Points ($iii$) and ($iv$) follow now straightforwardly.
\end{proof}

\section{Existence of normalized solutions for competing system}\label{sec: existence}

This section is devoted to the proof of Theorem \ref{thm: beta<0}. The proof is divided into three main steps: in the first part we study some useful properties of the unique radial ground state solution of the scalar Schr\"odinger equation. With these, we prove the existence of a Palais-Smale sequence for $J$ constrained on $\mathcal{P}$, which, by Theorem \ref{thm: natural}, provides a Palais-Smale sequence for $J$ on $S_{a_1} \times S_{a_2}$; in the last part of the proof, we discuss the convergence of the Palais-Smale sequence to a solution of \eqref{complete problem}.

\subsection{The ground state of the cubic Schr\"odinger equation}

In this section we consider general $a,\mu>0$. Let us introduce
\[
S_{a}^r := \left\{u \in S_{a}: \text{$u$ is radially symmetric with respect to $0$}\right\}.
\]
We denote by $w = w_{a,\mu}$ the unique function solving, for some $\nu<0$, the problem
\begin{equation}\label{def w}
\begin{cases}
-\Delta w -\nu w = \mu w^3 & \text{in $\R^3$} \\
w >0 & \text{in $\R^3$} \\
w \in S_{a}^r
\end{cases}
\end{equation}
(we refer e.g. to \cite[Proposition 2.2]{BaJeSo} for existence, uniqueness, and basic properties of $w$). From the variational point of view, $w$ is characterized as a mountain pass critical point of the functional
\begin{equation}\label{scalar functional}
I(u) = I_{\mu}(u):= \int_{\R^3} \frac{1}{2} |\nabla u|^2 - \frac{\mu}{4} u^4
\end{equation}
on $S_a^r$, and is the \emph{least energy solution} of the problem (i.e. the solution having minimal energy among all the nontrivial solutions). The energy level $I(w)$ is called \emph{ground state level}, and is denoted by $\ell(a,\mu)$.

Let us introduce
\begin{equation}\label{scalar manifold}
\mathcal{M}= \mathcal{M}_{a,\mu}:= \left\{ u \in S_a: \int_{\R^3} |\nabla u|^2 = \frac{3\mu }{4} \int_{\R^3} u^4 \right\}.
\end{equation}
It is not difficult to modify the proof of Lemma \ref{lem: manifold} (alternatively, one can directly apply the forthcoming Lemma \ref{lem: single manifold}) to check that $\mathcal{M} \cap S_a^r$ is a $\mathcal{C}^1$-submanifold of $S_a^r$, so that $w$ is a critical point of $I$ on $\mathcal{M}$.

\begin{lemma}\label{PS on M cubic}
The Palais-Smale condition holds for $I$ restricted to $\mathcal{M}$.
\end{lemma}

\begin{proof}
We refer the reader to the more general Lemma \ref{PS on M}.
\end{proof}


\begin{proposition}\label{prop: uniqueness minimizer}
The function $w$ is the unique positive radial minimizer for $I$ on $\mathcal{M}$. The set of minimizers for $I$ on $\mathcal{M}$ is $\{w,-w\}$.
\end{proposition}

\begin{proof}
The minimality of $w$ is proved in \cite[Proposition 8.2.4]{Caz} or \cite[Lemma 2.10]{Jea}. For the uniqueness, by Theorem \ref{thm: constraint single intro} we know that any minimizer $v$ of $I$ on $\mathcal{M}$ yields a solution $(\nu,v)$ to \eqref{def w}, and hence the uniqueness of $w$ as positive minimizer follows by \cite[Proposition 2.2]{BaJeSo}. Notice that also $-w$ is a minimizer. In order to prove that no-sign-changing minimizer exists, we argue by contradiction noting that if $v$ minimizes $I$ on $\mathcal{M}$, so does $|v|$. Thus, by Theorem \ref{thm: constraint single intro} and the previous lemma, $|v|$ is a non-negative solution to \eqref{def w} for some $\nu \in \R$, and its zero-level set $\{|v|=0\}$ is not empty. The strong maximum principle implies then that $|v| \equiv 0$, which is impossible since $0 \not \in \mathcal{M}$.
\end{proof}


\begin{lemma}\label{lem: s_u scalar}
For any $u \in S_a^r$ there exists a unique $s_u \in \R$ such that $s_u \star u \in \mathcal{M}$. It is defined by the equation
\[
e^{s_u} = \frac{4 \int_{\R^3} |\nabla u|^2}{3 \int_{\R^3} \mu u^4}.
\]
The value $s_u$ is the unique (strict) maximum point of the function $s \mapsto I(s \star u)$.
\end{lemma}

\begin{proof}
Existence and uniqueness are contained in \cite[Lemma 2.9]{Jea}. The explicit expression of $s_u$ follows by direct computations, observing that
\[
I(s \star u) = \frac{e^{2s}}{2} \int_{\R^3} |\nabla u|^2 - \frac{\mu e^{3s}}{4} \int_{\R^3} u^4. \qedhere
\]
\end{proof}

We conclude this section with the simple observation that $0$ is a critical point of the functional $I$ extended to the whole space $H^1(\R^3,\R^2)$, and the free second differential of $I$ in $0$ is positive.

\begin{lemma}\label{lem: 0 minimal}
If we consider $I$ as a functional in $H^1(\R^3,\R^2)$, we have that
\[
d I(0) = 0 \quad \text{and} \quad d^2 I(0)[\varphi,\varphi] =  \int_{\R^3} |\nabla \varphi|^2.
\]
for any $\varphi \in H^1(\R^3,\R^2)$.
\end{lemma}

\subsection{Construction of a Palais-Smale sequence for $J|_{\mathcal{P}}$}

We aim at proving that $J|_{\mathcal{P}}$ satisfies the assumptions of a minimax principle, and more precisely it has a mountain pass geometry. Our argument is somehow inspired by the proof of Theorem 5.4 in \cite{AmCo}, even though this result is tailor-made for the case $\beta>0$. Here several complications arise because we deal with $3$ constraints (and not with $1$) and with arbitrary $\beta<0$.

First, having in mind that the compactness of any Palais-Smale sequence would be far from being trivial, we confine ourselves in a radial setting. That is, we work in $S_{a_1}^r \times S_{a_2}^r$ instead of in $S_{a_1} \times S_{a_2}$. Since the problem is rotation-invariant, this is possible as a consequence of the principle of symmetric criticality \cite{Pa}.

For $i=1,2$, consider $w_i := w_{a_i,\mu_i}$ as defined in \eqref{def w}. We recall that $w_i$ is a minimizer for $I_i:= I_{\mu_i}$ in $\mathcal{M}_i:= \mathcal{M}_{a_i,\mu_i}$ (see \eqref{scalar functional} and \eqref{scalar manifold}). This suggests that $(w_1,0)$ and $(0,w_2)$ are local minimizers for $J$ on $\mathcal{P}$ (recall \eqref{def energy} and \eqref{definition Pohozaev}), so that $J|_{\mathcal{P}}$ has a mountain pass geometry. Of course, such an argument is incorrect in the present setting, since for instance $(w_1,0)$ and $(0,w_2)$ do not belong to $\mathcal{P} \subset S_{a_1} \times S_{a_2}$, or else the set $\mathcal{P}$ is not necessarily connected by arcs for $\beta < -\sqrt{\mu_1 \mu_2}$. On the other hand, in what follows we show how the previous heuristic idea can be adjusted in our context, leading to the following statement:

\begin{proposition}\label{prop: existence PS}
There exists a Palais-Smale sequence $\{(\tilde u_{1,n},\tilde u_{2,n})\}$ at a mountain pass level
\[
c > \max\{\ell(a_1,\mu_1), \ell(a_2,\mu_2)\}
\]
for $J$ restricted to $\mathcal{P}$, satisfying the additional condition $\tilde u_{i,n}^- \to 0$ a.e. in $\R^3$ for $i=1,2$. As a consequence, there exists a Palais-Smale sequence $\{(u_{1,n}, u_{2,n})\}$ for $J$ on $S_{a_1} \times S_{a_2}$, with $(u_{1,n}, u_{2,n}) \in \mathcal{P}$ for every $n$, such that $u_{i,n}^- \to 0$ a.e. in $\R^3$ for $i=1,2$. \\
Moreover, there exists $C>0$ independent of $\beta$ such that $c<C$.
\end{proposition}

Recall that $\ell(a_i,\mu_i)$ denotes the ground state energy level $I_i(w_i)$. Without loss of generality, we can suppose that
\begin{equation}\label{hp sui livelli}
\ell(a_1,\mu_1) \ge \ell(a_2,\mu_2).
\end{equation}
We have already mentioned that $(w_1,0),(0,w_2) \not \in S_{a_1}^r \times S_{a_2}^r$. On the other hand they both belong to the closure, with respect to the $\mathcal{D}^{1,2}$ topology, of $S_{a_1}^r \times S_{a_2}^r$, where as usual
\[
\mathcal{D}^{1,2}:= \left\{ u \in L^6(\R^3): \int_{\R^3} |\nabla u|^2 <+\infty \right\}, \quad \|u\|_{\D^{1,2}}^2 := \int_{\R^3} |\nabla u|^2.
\]
Actually, we can easily check that any family of type $\{(w_1, s \star v): s \in \R\}$, with $v \in S_{a_2}$, strongly converges in $\mathcal{D}^{1,2}$, as $s \to -\infty$, to $(w_1,0)$. It is sufficient to observe that
\[
\|s \star v\|_{\mathcal{D}^{1,2}}^2 = \int_{\R^3} |\nabla (s \star v)|^2 = e^{2s} \int_{\R^3} |\nabla v|^2 \to 0
\]
as $s \to -\infty$. In particular, this implies that
\[
\left( B(w_1,\rho_1; H^1) \times B(0,\rho_2; \mathcal{D}^{1,2}) \right) \cap \left( S_{a_1}^r \times S_{a_2}^r \right) \neq \emptyset
\]
for any $\rho_1,\rho_2 >0$, where $B(u,\rho; F)$ denotes the ball in $F$ (Banach space) with centre $u$ and radius $\rho$.

Before proceeding, we also emphasize that in the previous example there is no strong convergence in $H^1(\R^3,\R^2)$, since $\|s \star v\|_{L^2(\R^3)} = a_2$ for every $s$. This phenomenon is related to the fact that weak convergence in $H^1_{\textrm{rad}}(\R^3)$ does not imply strong convergence in $L^2(\R^3)$, and is a source of complications when dealing with normalization constraints of type \eqref{normalization}.

In order to prove Proposition \ref{prop: existence PS}, we investigate $s_{u_1}$, defined in Lemma \ref{lem: s_u scalar}, for $(u_1,u_2) \in \left( B(w_1,\rho_1; H^1) \times B(0,\rho_2; \mathcal{D}^{1,2}) \right) \cap  \mathcal{P}$ and determine the asymptotic behaviour when $\rho_1,\rho_2 \to 0$.


\begin{lemma}\label{bar s of rho}
There exist $\delta_1>0$ small and $C>0$ such that
\[
0<\rho_1,\rho_2<\delta_1 \quad \Longrightarrow \quad |s_{u_1}| \le C \rho_2^{3/2},
\]
for every $(u_1,u_2) \in \left( B(w_1,\rho_1; H^1)  \times B(0,\rho_2;\mathcal{D}^{1,2}) \right)  \cap \mathcal{P}$.
\end{lemma}

\begin{proof}
On one side, as $(u_1,u_2) \in \mathcal{P}$,
\begin{equation}\label{u su P}
1 = \frac{4 \int_{\R^3} \sum_i |\nabla u_i|^2}{3 \int_{\R^3} \sum_{i,j} \beta_{ij} u_i^2 u_j^2}.
\end{equation}
On the other hand, by the Lagrange theorem there exists $\xi \in (0,1)$ such that
\begin{equation}\label{19024}
\begin{split}
 \frac{4 \int_{\R^3} \sum_i |\nabla u_i|^2}{3 \int_{\R^3} \sum_{i,j} \beta_{ij} u_i^2 u_j^2}
&= \frac{4 \int_{\R^3} |\nabla u_1|^2}{3 \int_{\R^3}  \beta_{11} u_1^4} \\
& \quad + \frac{4\int_{\R^3}|\nabla u_2|^2}{ 3\left(\int_{\R^3}  \beta_{11} u_1^4 + \xi \int_{\R^3} 2\beta_{12} u_1^2 u_2^2 + \beta_{22} u_2^4\right)}\\
& \quad - \frac{4\left( \int_{\R^3} |\nabla u_1|^2 + \xi |\nabla u_2|^2\right) \left( \int_{\R^3} 2\beta_{12} u_1^2 u_2^2 + \beta_{22} u_2^4\right)}{ 3\left(\int_{\R^3}  \beta_{11} u_1^4 + \xi \int_{\R^3} 2\beta_{12} u_1^2 u_2^2 + \beta_{22} u_2^4\right)^2}
\end{split}
\end{equation}
The first term on the right hand side is, by definition, $\exp(s_u)$ (see Lemma~\ref{lem: s_u scalar}). In order to estimate the remaining terms on the right hand side, we recall the Gagliardo-Nirenberg inequality: there exists $S>0$ such that
\[
\int_{\R^3} w^4 \le S \left( \int_{\R^3} w^2 \right)^{1/2} \left(\int_{\R^3} |\nabla w|^2 \right)^{3/2} \qquad\text{for all } w \in H^1(\R^3).
\]
Let $\rho_1$ and $\rho_2$ small so that
\begin{equation}\label{choice rho}
\rho_2  \le \min\left\{1,\int_{\R^3}|\nabla w_1|^2, \frac{1}{4} \int_{\R^3} w_1^4 \right\} \le \frac12\int_{\R^3} u_1^4 \le \int_{\R^3} w_1^4
\end{equation}
for $\|u_1-w_1\|_{H^1} < \rho_1$. For $(u_1,u_2) \in \left(B(w_1,\rho_1; H^1) \times B(0,\rho_2;\D^{1,2})\right) \cap \mathcal{P}$ and $\xi \in (0,1)$, we have
\[
\int_{\R^3} |\nabla u_1|^2 + \xi |\nabla u_2|^2 \le 3 \int_{\R^3} |\nabla w_1|^2,
\]
and
\begin{align*}
\left| \int_{\R^3} \beta_{22} u_2^4 + 2\beta_{12} u_1^2 u_2^2 \right|  & \le \beta_{22} S a_2 \|u_2\|_{\mathcal{D}^{1,2}}^3 + 2|\beta_{12}| \left(\int_{\R^3} u_1^4 \right)^{1/2}\left(\int_{\R^3} u_2^4 \right)^{1/2}  \\
& \le \beta_{22} S a_2 \|u_2\|_{\mathcal{D}^{1,2}}^3 + \sqrt{8 S a_2} |\beta_{12}| \|w_1\|_{L^4}^{1/2}  \|u_2\|_{\D^{1,2}}^{3/2} \\
& \le C \|u_2\|_{\mathcal{D}^{1,2}}^{3/2},
\end{align*}
hence, replacing $\rho_2$ with a smaller quantity if necessary,
\[
\bigg|\int_{\R^3}  \beta_{11} u_1^4   + \xi \int_{\R^3} 2\beta_{12} u_1^2 u_2^2 + \beta_{22} u_2^4\bigg|
 \ge \frac12 \int_{\R^3}  \beta_{11} w_1^4-C \rho_2^{3/2}  \ge \frac{1}{4} \int_{\R^3} \beta_{11} w_1^4.
\]
As a consequence we obtain
\begin{equation}\label{19022}
\begin{aligned}
&\left| \frac{\int_{\R^3}|\nabla u_2|^2}{ \left(\int_{\R^3}  \beta_{11} u_1^4 + \xi \int_{\R^3} 2\beta_{12} u_1^2 u_2^2 + \beta_{22} u_2^4\right)}\right|
&\hspace{.5cm}
\le \frac{4 \|u_2\|_{\mathcal{D}^{1,2}}^2}{ \int_{\R^3} \beta_{11} w_1^4}   \le  C \rho_2^2,
\end{aligned}
\end{equation}
and
\begin{multline}\label{19023}
 \left|\frac{\left( \int_{\R^3} |\nabla u_1|^2 + \xi |\nabla u_2|^2\right) \left( \int_{\R^3} 2\beta_{12} u_1^2 u_2^2 + \beta_{22} u_2^4\right)}{ \left(\int_{\R^3}  \beta_{11} u_1^4 + \xi \int_{\R^3} 2\beta_{12} u_1^2 u_2^2 + \beta_{22} u_2^4\right)^2}\right| \\
 \le C\frac{\left( \int_{\R^3} |\nabla w_1|^2\right) \|u_2\|_{\D^{1,2}}^{3/2}}{\int_{\R^3} \beta_{11} w_1^4 }
 \le C \rho_2^{3/2}.
\end{multline}
Using \eqref{u su P}, \eqref{19022} and \eqref{19023}, we see that \eqref{19024} becomes
\[
1= e^{s_{u_1}} + O\left( \rho_2^{3/2}\right),
\]
which implies the lemma.
\end{proof}

\begin{lemma}\label{lem: 0104}
For any $r_1,r_2 > 0$ there exists $\delta_2 \in (0,\delta_1]$ such that if $\rho_1,\rho_2 \in (0,\delta_2)$, then
\[
\|s_{u_1} \star u_1-w_1\|_{H^1} < r_1 \quad \text{and} \quad \|s_{u_1} \star u_2\|_{\D^{1,2}} < r_2
\]
for every $(u_1,u_2) \in \left( B(w_1,\rho_1; H^1)  \times B(0,\rho_2;\mathcal{D}^{1,2}) \right)  \cap \mathcal{P}$.
\end{lemma}

\begin{proof}
Let $0<\rho_1,\rho_2<\delta_1$. First of all, by Lemma~\ref{bar s of rho} and Lemma~\ref{lem: s strong convergence} there exists $\delta' \in (0,\delta_1]$ so that $0<\rho_2<\delta'$ implies
\[
\|s_{u_1} \star w_1-w_1\|_{H^{1}} < \frac{r_1}{2}
\]
for every $(u_1,u_2) \in \left( B(w_1,\rho_1; H^1) \times B(0,\rho_2;\mathcal{D}^{1,2}) \right)  \cap \mathcal{P}$. Now we observe that for $\rho_1 \in (0,\delta_1)$ and $\rho_2 \in (0,\delta')$
\begin{align*}
\|s_{u_1} \star u_1-w_1\|_{H^1} & \le \|s_{u_1} \star u_1- s_{u_1} \star w_1\|_{H^1} + \| s_{u_1} \star w_1-w_1\|_{H^1} \\
& \le \max\{e^{s_{u_1}},1\} \|u_1-w_1\|_{H^1} + \frac{r_1}{2} \\
&  \le e^{C \rho_2^{3/2}} \rho_1 + \frac{r_1}{2} < r_1
\end{align*}
provided $\rho_1$ and $\rho_2$ are small enough, and similarly
\[
\|s_{u_1} \star u_2\|_{\D^{1,2}} = e^{s_{u_1}} \| u_2\|_{\D^{1,2}} \le e^{C \rho_2^{3/2}} \rho_2 < r_2
\]
for $\rho_2$ small.
\end{proof}

\begin{lemma}\label{lem: last quantification}
Let $\rho_1 \in (0, \delta_2)$ be fixed. There exists $\delta_3 \in (0,\delta_2]$ (possibly depending on $\rho_1$) such that
\[
\inf \left\{ \| s_{u_1} \star u_1- w_1 \|_{H^1} \left| \begin{array}{l} (u_1,u_2) \in \mathcal{P} \\
u_1 \in \pa B(w_1,\rho_1; H^1)  \\
u_2 \in  B(0,\rho_2;\mathcal{D}^{1,2}) \end{array}\right. \right\} > \frac{\rho_1}{2},
\]
and
\[
\inf \left\{ \| s_{u_1} \star u_2 \|_{\D^{1,2}} \left| \begin{array}{l} (u_1,u_2) \in \mathcal{P} \\
u_1 \in  B(w_1,\rho_1; H^1)  \\
u_2 \in  \pa B(0,\rho_2;\mathcal{D}^{1,2}) \end{array}\right. \right\} > \frac{\rho_2}{2}
\]
for every $\rho_2 \in (0,\delta_3)$.
\end{lemma}
\begin{proof}
We start with the first estimate in the thesis. Let us suppose by contradiction that there exist sequences $\rho_{2,n} \to 0$ and
\[
(u_{1,n},u_{2,n}) \in \left( \pa B(w_1,\rho_1; H^1)  \times B(0,\rho_{2,n};\mathcal{D}^{1,2}) \right)  \cap \mathcal{P},
\]
such that $s_n:= s_{u_{1,n}}$ satisfies
\[
\| s_n \star u_{1,n}- w_1 \|_{H^1} \le \rho_1/2 \qquad \forall n.
\]
First, from Lemma \ref{bar s of rho} we deduce that $s_n \to 0$ as $n \to \infty$, and hence by Lemma \ref{lem: s strong convergence} we have $s_n \star w_1 \to w_1$ strongly in $H^1(\R^3)$ as $n \to \infty$.

Now it is not difficult to obtain a contradiction, using again the fact that $s_n \to 0$:
\begin{align*}
\| s_n \star u_{1,n}- w_1 \|_{H^1}& \ge \|s_n \star u_{1,n}- s_n \star w_1 \|_{H^1}- \| s_n \star w_{1}- w_1 \|_{H^1} \\
& = e^{s_n } \|\nabla(u_{1,n}-w_1)\|_{L^2} + \| u_{1,n}-w_1\|_{L^2}- o(1)  \to \rho_1
\end{align*}
as $n \to \infty$, a contradiction.

For the second estimate in the thesis, we use again Lemma \ref{bar s of rho}:
\[
\| s_{u_1} \star u_2\|_{\D^{1,2}} = e^{s_{u_1}} \|u_2\|_{\D^{1,2}} \ge  e^{-C \rho_2^{3/2}} \rho_2 > \frac{\rho_2}{2}
\]
whenever $u_2 \in \pa B(0,\rho_2; \D^{1,2})$, provided $\rho_2>0$ is small enough.
\end{proof}

\begin{lemma}\label{lem: minimax inequality}
There exist $\rho_1,\rho_2, \bar C> 0$ such that
\begin{align*}
J(u_1,u_2) &\ge \ell(a_1,\mu_1) + \bar C
\end{align*}
for every $(u_1,u_2) \in  \pa \left( B(w_1,\rho_1; H^1) \times B(0,\rho_2; \mathcal{D}^{1,2}) \right) \cap  \mathcal{P}$.
\end{lemma}

\begin{proof}
Let $(u_1,u_2) \in \mathcal{P}$, and recall that $s_{u_1} \star u_1 \in \mathcal{M}_1$, with $\mathcal{M}_1$ defined in \eqref{scalar manifold}.

As a consequence of Lemma \ref{lem: structure P}, and using $\beta \le 0$, we obtain
\begin{equation}\label{0404}
J( u_1,u_2) \ge J(s_{u_1} \star (u_1,u_2)) \ge I_1(s_{u_1} \star u_1) + I_2(s_{u_1} \star u_2).
\end{equation}
To estimate the right hand side, we observe that by Lemma \ref{lem: 0 minimal} there exist $\bar \rho_2>0$ and $C>0$ such that
\begin{equation}\label{04041}
I(v_2) \ge C\|v_2\|_{\mathcal{D}^{1,2}}^2 \qquad  \text{if $v_2 \in B(0,\bar \rho_2;\mathcal{D}^{1,2})$}.
\end{equation}
Let $\bar \rho_1>0$ be such that $\bar \rho_1 < \| w_1\|_{H^1}$. By Lemma \ref{lem: 0104} there exists $\delta_2 >0$ such that $\rho_1,\rho_2 \in (0,\delta_2)$ implies
\[
s_{u_1} \star (u_1,u_2) \in B(w_1,\bar \rho_1; H^1) \times B(0,\bar \rho_2;\mathcal{D}^{1,2})
\]
provided $(u_1,u_2) \in  \left( B(w_1,\rho_1; H^1) \times B(0,\rho_2;\mathcal{D}^{1,2}) \right) \cap \mathcal{P}$. Now we fix $\rho_1 \in (0,\delta_2)$, and next $\rho_2 \in (0,\delta_3]$, with $\delta_3>0$ given by Lemma \ref{lem: last quantification}. We claim that $\rho_1$ and $\rho_2$ are the desired quantities. To prove the claim, we observe first that the boundary of $B(w_1,\rho_1; H^1) \times B(0,\rho_2;\mathcal{D}^{1,2})$ splits as
\[
\left[\pa B(w_1,\rho_1; H^1) \times B(0,\rho_2; \mathcal{D}^{1,2})\right] \cup \left[   B(w_1,\rho_1; H^1) \times \pa B(0,\rho_2; \mathcal{D}^{1,2})\right].
\]
Since $s_{u_1}\star u_1 \in \mathcal{M}_1$ and $\ell(a_1,\mu_1) = \inf_{\mathcal{M}_1} I_1$, we have by \eqref{04041} and Lemma \ref{lem: last quantification} (notice that the lemma is applicable in light of our choice of $\rho_1$ and $\rho_2)$)
\begin{equation}\label{18041}
\begin{split}
J(u_1,u_2) & \ge I_1(s_{u_1} \star u_1) + I_2(s_{u_1} \star u_2) \\
& \ge \ell(a_1,\mu_1) + C \|s_{u_1} \star u_2\|_{\D^{1,2}}^2 \ge \ell(a_1,\mu_1) + C \frac{\rho_2^2}{4}
\end{split}
\end{equation}
for every $(u_1,u_2) \in  \left( B(w_1,\rho_1; H^1) \times \pa B(0,\rho_2;\mathcal{D}^{1,2}) \right) \cap \mathcal{P}$. On the other hand, using again \eqref{04041} we deduce also that
\[
J(u_1,u_2) \ge I(s_{u_1} \star u_1).
\]
for every $(u_1,u_2) \in  \left( \pa B(w_1,\rho_1; H^1) \times  B(0,\rho_2;\mathcal{D}^{1,2}) \right) \cap \mathcal{P}$. We claim that
\begin{equation}\label{claim 1804}
\inf \left\{ I_1(s_{u_1} \star u_1): (u_1,u_2) \in  \left( \pa B(w_1,\rho_1; H^1) \times  B(0,\rho_2;\mathcal{D}^{1,2}) \right) \cap \mathcal{P}\right\} > \ell(a_1,\mu_1).
\end{equation}
Indeed, the Lemmas \ref{lem: 0104} and \ref{lem: last quantification} yield
\[
\frac{\rho_1}{2} \le \|s_{u_1} \star u_1-w_1\|_{H^1} < \bar \rho_1
\]
for every $(u_1,u_2) \in  \left( \pa B(w_1,\rho_1; H^1) \times  B(0,\rho_2;\mathcal{D}^{1,2}) \right) \cap \mathcal{P}$, so that the left hand side in \eqref{claim 1804} is larger than or equal to
\[
\inf \left\{ I_1(u) : \frac{\rho_1}{2} \le \|u-w_1\|_{H^1} < \bar \rho_1, \ u \in \mathcal{M}_1 \right\}.
\]
If by contradiction this infimum is $\ell(a_1,\mu_1)$, then there exists a bounded sequence $\{u_n\} \subset \mathcal{M}_1$ with $\|u_n -w_1\|_{H^1} \ge \rho_1/2$ such that $I_1(u_n) \to \ell(a_1,\mu_1)$; that is, $\{u_n\}$ is a bounded minimizing sequence for $I_1$ restricted to $\mathcal{M}_1$. By Lemma \ref{PS on M cubic} we infer that $u_n \to u$ strongly in $H^1(\R^3)$, where by strong convergence $u$ minimizes $I_1$ on $\mathcal{M}_1$. Notice that
\[
\|u-w_1\|_{H^1} \le \bar \rho_1 \le \|w_1\|_{H^1} < \|w_1-(-w_1)\|_{H^1};
\]
this rules out the possibility that $u=-w_1$, so that by Proposition \ref{prop: uniqueness minimizer} necessarily $u = w_1$. But on the other hand, always by strong convergence, $\|u-w_1\|_{H^1} \ge \rho_1/2$, a contradiction. This proves claim \eqref{claim 1804}, which together with \eqref{18041} and \eqref{0404} gives the thesis.
\end{proof}

In order to complete the proof of Proposition \ref{prop: existence PS}, the idea is now to define a convenient minimax class of paths connecting two pairs $(u_1,u_2)$ and $(v_1,v_2)$, sufficiently close to $(w_1,0)$ and to $(0,w_2)$ respectively. The problem is that, at least for $\beta <-\sqrt{\mu_1 \mu_2}$, it is not clear whether the set $\mathcal{P}$ is connected by arcs, and in particular it is not clear if an arc connecting $(u_1,u_2)$ and $(v_1,v_2)$ does exists. In the next lemma we conveniently choose $(u_1,u_2)$ and $(v_1,v_2)$ so that they lie in the same connected component of $\mathcal{P}$.

\begin{lemma}\label{lem: estremi}
Let $\rho_1,\rho_2$ be defined in Lemma \ref{lem: minimax inequality}. For every $\eps>0$ there exist
\begin{align*}
(u_1^\eps,u_2^\eps) \in \left( B(w_1,\rho_1; H^1) \times B(0,\rho_2; \mathcal{D}^{1,2}) \right) \cap \mathcal{P}\\
(v_1^\eps,v_2^\eps) \in \left( B(0,\rho_1; \D^{1,2}) \times B(w_2,\rho_2; H^1) \right) \cap \mathcal{P}
\end{align*}
with the following properties:
\begin{itemize}
\item[($i$)] $(u_1^\eps,u_2^\eps), (v_1^\eps,v_2^\eps) \in \mathcal{C}^\infty_c(\R^3,\R^2)$ and $u_i^\eps, v_i^\eps \ge 0$ in $\R^3$ for both $i=1,2$.
\item[($ii$)] $u_1^\eps u_2^\eps \equiv 0$,  $v_1^\eps v_2^\eps \equiv 0$, and $u_2^\eps v_1^\eps \equiv 0$.
\item[($iii$)] $J(u_1^\eps,u_2^\eps), J(v_1^\eps,v_2^\eps) \le \ell(a_1,\mu_1) + \eps$.
\item[($iv$)] There exists $\gamma=(\gamma_1,\gamma_2):[0,1] \to \mathcal{P}$, continuous with respect to the $H^1$-topology, such that $\gamma(0)=(u_1^\eps,u_2^\eps)$, $\gamma(1) = (v_1^\eps,v_2^\eps)$, and moreover $\gamma_1(t) \gamma_2(t) \equiv 0$ in $\R^N$ for every $t \in [0,1]$.
\end{itemize}
\end{lemma}
\begin{proof}
Let $\eps>0$ be fixed. Arguing as in the proof of Theorem \ref{thm: natural}-($ii$), we can check that there exists $\{u_{1,n}\} \subset \mathcal{C}^\infty_c(\R^3) \cap \mathcal{M}_1$ strongly convergent to $w_1$ in $H^1(\R^3)$ as $n \to \infty$; moreover, since $w_1 > 0$ in $\R^3$, it is not restrictive to suppose that $u_{1,n} \ge 0$ in $\R^3$ for every $n$ sufficiently large. By continuity, we can take $u_{1,\bar n}$ with $\bar n$ very large, so that
\begin{equation}\label{def u bar n}
I_1(u_{1,\bar n}) < \ell(a_1,\mu_1) + \frac{\eps}{2}.
\end{equation}
The support of $u_{1,\bar n}$ is contained in $B_R(0)$ for some positive $R>0$.

Let us now consider $u \in S_{a_2}^r$, $u \ge 0$ in $\R^3$, with support in $A(0;2,3)$, the annulus of center $0$ and radii $2<3$, and define
\[
u_{2,m}(x):=((-m) \star u)(x) = e^{-3m/2} u(e^{-m}x).
\]
Then $u_{2,m} \to 0$ strongly in $\D^{1,2}$ as $m \to \infty$, and $\supp u_{2,m} \subset A(0; 2e^{m},3e^{m})$, as
\begin{align*}
\supp u_{2,m} = \{ 2<|e^{-m} x|<3 \} = \{ 2 e^{m} < |x| < 3 e^m\} \\
\| u_{2,m} \|_{\D^{1,2}}^2 = e^{-2m} \|u\|_{\D^{1,2}}^2 \to 0 \qquad \text{as $m \to \infty$}.
\end{align*}
In particular, there exists $\bar m$ very large so that
\[
\text{$\supp u_{2,m} \cap B_R(0) =\emptyset$ for every $m \ge \bar m$}.
\]
Let $s_m:= s_{(u_{1,\bar n},u_{2,m})}$ be defined by Lemma \ref{lem: structure P} (we remark that $s_m$ is well defined, since by construction $(u_{1,\bar n},u_{2,m}) \in \mathcal{E}$, with $\mathcal{E}$ defined in \eqref{def cone}). Since
\[
u_{1,\bar n} \in \mathcal{M}_1 \quad \Longrightarrow \quad \frac{4 \int_{\R^3} |\nabla u_{1,\bar n}|^2}{3 \int_{\R^3} \beta_{11} u_{1,\bar n}^4} = 1
\]
and $u_{2,m} \to 0$ in $\D^{1,2}$, it is possible to repeat step by step (with minor changes) the computations between \eqref{19024} and \eqref{19023}, obtaining
\[
s_m = O(\|u_{2,m}\|_{\D^{1,2}}^{3/2}) \qquad \text{as $m \to \infty$}.
\]
Now, $s_m \star u_{2,m} \to 0$ strongly in $\mathcal{D}^{1,2}$ as $m \to +\infty$, since
\[
\|s_m \star u_{2,m}\|_{\D^{1,2}}= e^{s_m} \|u_{2,m}\|_{\D^{1,2}}.
\]
Therefore, by continuity (with respect to the $\D^{1,2}$ topology)
\[
I_2(s_m \star u_{2,m}) < \frac{\eps}{2}
\]
for any $m \ge  \tilde m$, with $\tilde m \ge \bar m$ sufficiently large. Observing that by construction $s_m\star u_{1,\bar n}$ and $s_m \star u_{2,m}$ have disjoint support,
that $I_1(s_m \star u_{1,\bar n}) < I_1(u_{1,\bar n})$ (see Lemma \ref{lem: s_u scalar}), and recalling \eqref{def u bar n}, we finally conclude 
\[
\begin{split}
J(s_m \star (u_{1,\bar n},u_{2,m})) &= I_1(s_m \star u_{1,\bar n}) + I_2(s_m \star u_{2,m})  \\
& < I_1(u_{1,\bar n}) +  \frac{\eps}{2}
 < \ell(a_1,\mu_1) + \eps
\end{split}
\]
for any $m \ge  \tilde m$. We set $(u_1^\eps,u_2^\eps) := \tilde m \star (u_{1,\bar n}, u_{2,\tilde m})$, and for future convenience we denote by $R_\eps$ a radius such that $\supp u_{2,\eps} \subset B_{R_\eps}(0)$.

\smallskip

The existence of $(v_1^\eps,v_2^\eps)$ can be proved in a similar way. We first take $v_{2,\bar n} \in \mathcal{C}^\infty_c(\R^3) \cap \mathcal{M}_2$ close to $w_2$ in $H^1$, supported in $B_{R'}(0)$ for some $R'>0$. Then we set $v_{1,m} := (-m) \star v$, where $v \in S_{a_1}^r$ is a function with $\supp v \subset A(0;2,3)$. There exists $\bar m$ so large that
\[
\supp v_{1,m} \cap B_{R'}(0)= \emptyset \quad \text{and also} \quad   \supp v_{1,m} \cap B_{R_\eps}(0)= \emptyset
\]
for every $m>\bar m$. Now we can proceed exactly as before, obtaining in the end a pair $(v_1^\eps,v_2^\eps)$ with
\[
J(v_1^\eps,v_2^\eps) < \ell(a_2,\mu_2) + \eps \le \ell(a_1,\mu_1) + \eps
\]
(recall condition \eqref{hp sui livelli}).

\smallskip

So far we proved the existence of $(u_1^\eps,u_2^\eps)$ and $(v_1^\eps,v_2^\eps)$ satisfying ($i$)-($iii$) of the thesis. It remains to prove ($iv$). To this purpose, it is sufficient to find a path $\tilde \gamma:[0,1] \to \mathcal{E}$ with $\tilde \gamma(0)=(u_1^\eps,u_2^\eps)$, $\tilde \gamma(1)=(v_1^\eps,v_2^\eps)$, and $\tilde \gamma_1(t)\tilde  \gamma_2(t) \equiv 0$ in $\R^N$ for every $t \in [0,1]$, where $\mathcal{E}$ was defined in \eqref{def cone}. Indeed if such a $\tilde \gamma$ does exist, then the path $\gamma(t) := s_{\tilde \gamma(t)} \star \tilde \gamma(t)$ satisfies all the properties in point ($iv$) of the lemma. For the continuity, we observe that $(u_{1,n},u_{2,n}) \to (u_1,u_2)$ in $H^1(\R^3)$ implies $s_n := s_{(u_{1,n},u_{2,n})} \to s_{(u_1,u_2)}=: s_\infty$. Thus, Lemma \ref{lem: s strong convergence} yields
\begin{align*}
\|  s_n \star u_{1,n} -s_\infty \star u_1 \|_{H^1} \to 0
\end{align*}
as $n \to \infty$, and the same holds for the second component.

In order to define $\tilde \gamma$, we set
\[
\sigma_1(t):= \left(a_1 \frac{(1-t)u_1^{\eps}  + t v_1^\eps}{\|(1-t)u_1^{\eps}  + t v_1^\eps\|_{L^2}}, u_2^\eps\right) \qquad t \in [0,1]
\]
Since ($ii$) of this lemma holds true, $\sigma_1(t) \in \mathcal{E}$ and $\sigma_{1,1}(t) \sigma_{1,2}(t) \equiv 0$ in $\R^N$ for every $t \in [0,1]$. Now we set
\[
\sigma_2(t):= \left( v_1^\eps,  a_2 \frac{(1-t)u_2^{\eps}  + t v_2^\eps}{\|(1-t)u_2^{\eps}  + t v_2^\eps\|_{L^2}} \right) \qquad t \in [0,1],
\]
and again we note that $\sigma_2(t) \in \mathcal{E}$ and $\sigma_{2,1}(t) \sigma_{2,2}(t) \equiv 0$ in $\R^N$ for every $t \in [0,1]$. The path
\[
\tilde \gamma(t):= \begin{cases} \sigma_1(2t) & t \in \left[0,\frac12\right] \\ \sigma_2(2t-1) & t \in \left[\frac12,1\right]
\end{cases}
\]
is then a continuous path on $\mathcal{E}$ connecting $(u_1^\eps,u_2^\eps)$ with $(v_1^\eps,v_2^\eps)$, and such that $\tilde \gamma_{1}(t) \tilde \gamma_{2}(t) \equiv 0$ in $\R^N$ for every $t \in [0,1]$.
\end{proof}

\begin{proof}[Proof of Proposition \ref{prop: existence PS}]
Let $\rho_1$, $\rho_2$ and $\bar C$ be defined in Lemma \ref{lem: minimax inequality}, and let $0<\eps< \bar C/2$.
For such an $\eps>0$, thanks to Lemma \ref{lem: estremi} we find $(u_1^\eps,u_2^\eps)$ and $(v_1^\eps,v_2^\eps)$. Let now $\bar{\mathcal{P}}$ be the
connected component of $\mathcal{P}$ containing $(u_1^\eps,u_2^\eps)$ and $(v_1^\eps,v_2^\eps)$ (the existence of $\bar{\mathcal{P}}$ follows by Lemma \ref{lem: estremi}-($iv$)). Recalling Lemma \ref{lem: manifold}, we have that $\bar{\mathcal{P}}$ is a complete connected
$\C^1$ manifold without boundary. We introduce the minimax class
\[
\Gamma:= \left\{ \gamma \in \mathcal{C}\left([0,1],\bar{\mathcal{P}}\right), \gamma(0) =  (u_1^\eps,u_2^\eps), \ \gamma(1) = (u_2^\eps,u_2^\eps) \right\},
\]
and the associated minimax level
\[
c:= \inf_{\gamma \in \Gamma} \max_{ t \in [0,1]} J(\gamma(t)).
\]
It is clear that for any $\gamma \in \Gamma$ there exists $t \in (0,1)$ such that
\[
\gamma(t) \in \pa \left(B(w_1,\rho_1; H^1) \cap B(0,\rho_2; \D^{1,2} ) \right) \cap \mathcal{P},
\]
so that Lemma \ref{lem: minimax inequality} and the choice of $\eps$ permit to apply the minimax principle \cite[Theorem 3.2]{Ghou}: we deduce that for every minimizing sequence $\{\gamma_n\} \subset \Gamma$ for $c$, there exists a Palais-Smale sequence $\{(\tilde u_{1,n},\tilde u_{2,n})\}$ of $J$ on $\mathcal{P}$ at level $c$, such that
\begin{equation}\label{dist to 0}
\dist_{H^1}( (\tilde u_{1,n}, \tilde u_{2,n}), \gamma_n([0,1]) ) \to 0 \qquad \text{as $n \to \infty$}.
\end{equation}
Since $J$ and $G$ (see \eqref{def G}) are even and $(u_1^\eps,u_2^\eps)$ and $(v_1^\eps,v_2^\eps)$ have both non-negative components, we claim that it is not restrictive to suppose that $\gamma_{1,n}(t),\gamma_{2,n} (t)\ge 0$ a.e. in $\R^3$, for every $n$, for every $t \in [0,1]$. To prove the claim, we show that if $\gamma \in \Gamma$, then also $|\gamma|:= (|\gamma_{1}|,|\gamma_{2}|) \in \Gamma$. It is clear that $|\gamma|$ is continuous and $|\gamma(t)| \in \mathcal{P}$, but we have to prove the stronger assertion $|\gamma(t)| \in \bar{\mathcal{P}}$. Let us define, for $t \in [0,1]$,
\[
\sigma_1(\tau):= \gamma_n((1-\tau)t ) \quad \text{and} \quad \sigma_2(\tau):= |\gamma_n(\tau t)|,
\]
with $\tau \in [0,1]$. Setting
\[
\sigma(\tau,t):= \begin{cases} \sigma_1(\tau t) & \tau \in \left[0,\frac12\right] \\
\sigma_2((2\tau-1)t ) & \tau \in \left[\frac12,1\right],
\end{cases}
\]
we have a continuous path in $\mathcal{P}$ connecting $\gamma(t)$ with $|\gamma(t)|$. Hence $\gamma(t)$ and $|\gamma(t)|$ live in the same connected connected component of $\mathcal{P}$. Since this holds for every $t$, we conclude that $|\gamma| \in \Gamma$, as desired.

The fact that $\gamma_{1,n}(t), \gamma_{2,n}(t) \ge 0$ a.e. in $\R^3$, together with \eqref{dist to 0}, imply that $\tilde u_{i,n}^- \to 0$ a.e.\ in $\R^3$. The rest of the proposition follows now by Theorem \ref{thm: natural}, with the exception of the uniform boundedness of $c=c_\beta$ with respect to $\beta$. To this purpose, let us denote by $\mathcal{P}_\beta$ the natural constraint defined in \eqref{definition Pohozaev} for a prescribed value of $\beta$, and by $J_\beta$ the associated energy functional. Let us consider the path $\gamma$ constructed in Lemma \ref{lem: estremi}. Since $\gamma_1(t) \gamma_2(t) \equiv 0$ in $\R^N$ for every $t \in [0,1]$, we have that $\gamma(t) \in \mathcal{P}_\beta$ for every $t \in [0,1]$, for every $\beta<0$. As a consequence, by definition
\[
c_\beta \le  \max_{t \in [0,1]} J_\beta(\gamma(t)) =:C
\]
with $C>0$ independent of $\beta$. This completes the proof.
\end{proof}

\subsection{Convergence of the Palais-Smale sequence}

In order to complete the proof of Theorem \ref{thm: beta<0}, we have to show that the Palais-Smale sequence $\{( u_{1,n},u_{2,n})\}$ strongly converges in $H^1(\R^3,\R^2)$ to a couple $( \bar u_1,\bar u_2)$, solution of \eqref{system} for suitable $\bar \lambda_1$, $\bar \lambda_2<0$. Once that this is done, we observe that by strong convergence $(\bar u_1,\bar u_2)$ fulfills also \eqref{normalization}, and hence $(\bar u_1,\bar u_2,\lambda_1,\lambda_2)$ is a solution to \eqref{complete problem}, as desired.

For the strong convergence, we adapt the argument used in the last part of the proofs of Theorems 1.1 and 1.2 in \cite{BaJeSo}. Since $(u_{1,n},u_{2,n}) \in \mathcal{P}$, arguing as in \cite[Lemma 3.7]{BaJeSo} we deduce that $\{( u_{1,n},u_{2,n})\}$ is bounded in $H^1(\R^3,\R^2)$, and moreover there exists $C>0$ such that
\[
\int_{\R^3} |\nabla u_{1,n}|^2 + |\nabla u_{2,n}|^2 \ge C \qquad \text{for all $n$}.
\]
Hence, up to a subsequence $(u_{1,n},u_{2,n}) \wc (\bar u_1,\bar u_2)$ weakly in $H^1$, strongly in $L^4$, and a.e. in $\R^3$ (we recall that the embedding $H^1_{\rad}(\R^3) \hookrightarrow L^4(\R^3)$ is compact), and in particular $\bar u_1, \bar u_2 \ge 0$ a.e. in $\R^3$. Since $dJ |_{S_{a_1} \times S_{a_2}}(u_{1,n},u_{2,n}) \to 0$, by the Lagrange multipliers rule
there exist two sequences of real numbers $(\lambda_{1,n})$ and $(\lambda_{2,n})$ such that
\begin{multline}\label{weak sol}
\sum_i \int_{\R^3}  \nabla u_{i,n} \cdot \nabla \varphi_{i} -\lambda_{i,n} u_{i,n} \varphi \\
 - \sum_{i,j} \int_{\R^3} \beta_{ij} u_{i,n} u_{j,n} (u_{i,n} \varphi_j + \varphi_i u_{j,n}) = o(1) \|(\varphi_1,\varphi_2)\|_{H^1}
\end{multline}
for every $(\varphi_1, \varphi_2) \in H^1(\R^3,\R^2)$, with $o(1) \to 0$ as $n \to \infty$. For more details we refer to \cite[Lemma 3.2]{BarJea}. In light of \eqref{weak sol}, we can check as in \cite[Lemma 3.8]{BaJeSo} that up to a subsequence $\lambda_{i,n} \to \lambda_i \in \R$ for $i=1,2$, and at least one limit value, say $\lambda_1$, is strictly negative. Moreover, thanks to \cite[Lemma 3.9]{BaJeSo}, we know that if $\lambda_i<0$, then necessarily $u_{i,n} \to \bar u_i$ strongly in $H^1$. Hence, to complete the proof of the strong convergence $(u_{1,n},u_{2,n}) \to (\bar u_1,\bar u_2)$, it remains to show that also $\lambda_2$ is negative. In \cite{BaJeSo} we used in a decisive way the assumption $\beta>0$, and hence we have to modify our argument in the following way. First, we notice that by \eqref{weak sol} and weak convergence $(\bar u_1,\bar u_2)$ is a (weak, and by regularity classical) solution to
\begin{equation}\label{limit problem}
\begin{cases}
-\Delta \bar u_i -\lambda_i \bar u_i = \sum_{i,j} \beta_{ij} \bar u_i \bar u_j^2 & \text{in $\R^3$} \\
u_i \ge 0 & \text{in $\R^3$}.
\end{cases}
\end{equation}
Being $\lambda_1<0$, we deduce the following decay property for $\bar u_1$.

\begin{lemma}\label{exp decay}
There exists $\alpha,\gamma>0$ such that
\[
\bar u_1(x) \le \alpha e^{-\sqrt{1+ \gamma|x|^2}}  \qquad \text{for every $x \in \R^3$}.
\]
\end{lemma}

\begin{proof}
It is well known that radially symmetric $H^1$ continuous functions uniformly converge to $0$ as $|x| \to +\infty$, see e.g. \cite[Lemma 3.1.2]{BaSe}. Thus, we observe that
\[
-\Delta \bar u_1 + q(x) \bar u_1 = 0 \qquad \text{in $\R^3$},
\]
where
\[
q(x) = -\lambda_1 -\beta \bar u_1 \bar u_2^2 -\mu_1 \bar u_1^3 \ge \frac{|\lambda_1|}{2}   \qquad \text{for $|x| > M$},
\]
provided $M$ is sufficiently large. Let $\alpha >0$ to be determined, $\gamma \in (0,|\lambda_1|/2)$, and
\[
z(x):= \alpha e^{-\sqrt{1+\gamma|x|^2}}.
\]
By direct computations
\[
-\Delta z + \gamma z \ge 0 \qquad \text{in $\R^N$},
\]
so that
\[
-\Delta (z-\bar u_1) + \gamma(z-\bar u_1) \ge \left( \frac{|\lambda_1|}{2} - \gamma \right) \bar u_1 \ge 0 \qquad \text{for $|x|>M$}.
\]
We can also choose $\alpha$ so large that $z \ge \bar u_1$ for $|x| \le M$. Therefore, testing the previous inequality with $(z-\bar u_1)^-$, we deduce that $\bar u_1 \le z$ in $\R^N$.
\end{proof}

Now we focus on the equation satisfied by $\bar u_2$:
\[
\begin{cases}
-\Delta \bar u_2 = \lambda_2 \bar u_2 + \mu_2 \bar u_2^3 - c(x) \bar u_2 & \text{in $\R^3$} \\
\bar u_2 \ge 0 & \text{in $\R^3$},
\end{cases}
\]
with $c(x) = |\beta| \bar u_1^2(x) \ge 0$.
If $\lambda_2 \ge 0$, we have then that $\bar u_2$ satisfies in weak sense a problem of type
\begin{equation}\label{pb Liouv}
\begin{cases}
-\Delta w +c(x) w \ge 0 & \text{in $\R^3$} \\
w \ge 0 & \text{in $\R^3$} \\
w \in H^1(\R^3),
\end{cases} \quad \text{with} \quad 0 \le c(x) \le C e^{-C|x|}
\end{equation}
for some $C>0$. We need a Liouville-type theorem for this class of problems.

\begin{lemma}\label{lem: Liouville}
If $w$ satisfies \eqref{pb Liouv}, then $w \equiv 0$ in $\R^3$.
\end{lemma}
\begin{proof}
We suppose by contradiction that $w \not \equiv 0$, so that by the strong maximum principle $w >0$ in $\R^3$, and we set $v(x):= |x|^{-\alpha} = r^{-\alpha}$ with $1<\alpha \le 3/2$. Then
\[
-\Delta v+c(x) v \le \alpha ( 1-\alpha) r^{-\alpha-2} + C e^{-C r} r^{-\alpha} <0
\]
for every $r>r_0$ with $r_0$ sufficiently large. Since $w>0$ in $\R^3$, there exists $C_0>0$ such that
\[
\min_{\pa B_{r_0}(0)} w = C_0 r_0^{-\alpha}.
\]
Letting $\varphi:= w- C_0 r_0^{-\alpha}$, we claim that
\begin{equation}\label{claim liouv}
\varphi>0 \qquad \text{in $\R^3 \setminus B_{r_0}(0)$}.
\end{equation}
To prove the claim, we observe that
\[
\begin{cases}
-\Delta \varphi +c(x) \varphi \ge 0 & \text{in $\R^3 \setminus B_{r_0}(0)$}\\
\varphi \ge 0 & \text{on $\pa B_{r_0}(0)$}.
\end{cases}
\]
Then, we test the first equation with $\varphi^- \eta_R^2$, where $\{\eta_R\}$ is a family of $\mathcal{C}^\infty_c(\R^3)$ such that $\eta_R \equiv 1$ in $B_R(0)$, $\eta_R \equiv 0$ in $\R^3 \setminus B_{R+1}(0)$, and $|\nabla \eta_R| \le C$ for some positive constant $C$ independent of $R$. We obtain
\begin{multline*}
\int_{B_R(0) \setminus B_{r_0}(0) } |\nabla \varphi^-|^2 + c(x) (\varphi^-)^2 \le \int_{ \R^3 \setminus B_{r_0}(0) } |\nabla (\varphi^- \eta_R)|^2 + c(x) (\varphi^- \eta_R)^2 \\
\le \int_{B_{R+1}(0) \setminus B_{R}(0)} (\varphi^-)^2 |\nabla \eta_R|^2
\le C \int_{B_{R+1}(0) \setminus B_{R}(0)} (\varphi^-)^2,
\end{multline*}
and passing to the limit as $R \to +\infty$ we deduce, thanks to the $L^2$-integrability of $w$, that
\[
0 \le \int_{ \R^3 \setminus B_{r_0}(0) } |\nabla \varphi^-|^2 + c(x) (\varphi^-)^2 \le 0.
\]
Since $c \ge 0$ by assumption, claim \eqref{claim liouv} follows.

At this point it is not difficult to conclude. Indeed, on one side $w \in L^2(\R^3)$, but on the other side
\[
\int_{\R^3 \setminus B_{r_0}(0)} w^2 \ge \int_{\R^3 \setminus B_{r_0}(0)} C_0 |x|^{-2\alpha}\,dx = C\,C_0 \int_{r_0}^{+\infty} r^{2(1-\alpha)}\,dr = +\infty,
\]
since $\alpha \le 3/2$. This gives a contradiction and completes the proof.
\end{proof}

\begin{proof}[Conclusion of the proof of Theorem \ref{thm: beta<0}]
We observed that, if $\lambda_2<0$, then necessarily $u_{2,n} \to \bar u_2$ strongly in $H^1(\R^3)$, and hence the proof is complete. Let us suppose by contradiction that $\lambda_2 \ge 0$. Then by Lemma \ref{lem: Liouville} we deduce that $\bar u_2 \equiv 0$, so that $\bar u_1$ is radial and solves
\begin{equation}\label{eq: scalar norm}
\begin{cases}
-\Delta \bar u_1 -\lambda_1 \bar u_1 = \mu_1 \bar u_1^3 & \text{in $\R^3$} \\
\bar u_1>0 & \text{in $\R^3$} \\
\int_{\R^3} \bar u_1^2 = a_1^2.
\end{cases}
\end{equation}
We infer that $\bar u_1 \in \mathcal{M}_{a_1,\mu_1}$, and moreover,  by the uniqueness of the radial positive solution to \eqref{eq: scalar norm},
\[
I_{\mu_1}(\bar u_1) = \ell(a_1,\mu_1).
\]
Notice also that, as $\bar u_1 \in \mathcal{M}_{a_1,\mu_1}$,
\[
I_{\mu_1}(\bar u_1) = \frac{\mu_1}{8} \int_{\R^3} \bar u_1^4.
\]
In the same way, since $(u_{1,n},u_{2,n}) \in \mathcal{P}$,
\[
J(u_{1,n},u_{2,n}) = \frac{1}{8}  \sum_{i,j} \int_{\R^3} \beta_{ij} u_{i,n}^2 u_{j,n}^2.
\]
Thus, by the strong $L^4$-convergence $(u_{1,n},u_{2,n}) \to (\bar u_1,0)$, the level $c$ of the Palais-Smale sequence $\{(u_{1,n},u_{2,n})\}$ is
\begin{equation}\label{eq37}
\begin{split}
c &= \lim_{n \to \infty} J(u_n,v_n)
   = \lim_{n\to\infty} \frac{1}{8}  \sum_{i,j} \int_{\R^3} \beta_{ij} u_{i,n}^2 u_{j,n}^2\\
 & = \frac{\mu_1}{8} \int_{\R^3} \bar u_1^4 = I_{\mu_1}(\bar u_1) = \ell(a_1,\mu_1),
\end{split}
\end{equation}
in contradiction with the fact that $c> \ell(a_1,\mu_1)$ (see Proposition \ref{prop: existence PS}).
\end{proof}

%

\subsection{Phase-separation}\label{sec: phase sep}

In this subsection we prove Theorem \ref{thm: phase sep}, and we use the subscript $\beta$ to emphasize the dependence of all the considered quantities and functions with respect to $\beta$.

Due to the uniform bound $c_\beta = J_\beta(\bar u_{1,\beta},\bar u_{2,\beta}) \le C$ (see Proposition \ref{prop: existence PS}), the proof follows a well understood scheme. Since $(\bar u_{1,\beta},\bar u_{2,\beta}) \in \mathcal{P}_\beta$ for every $\beta$, we have
\[
J_\beta(\bar u_{1,\beta},\bar u_{2,\beta}) = \frac{1}{6} \int_{\R^N} |\nabla \bar u_{1,\beta}|^2 + |\nabla \bar u_{2,\beta}|^2,
\]
and hence $\{(\bar u_{1,\beta},\bar u_{2,\beta})\}$ is bounded in $H^1(\R^3,\R^2)$. Testing the equation \eqref{system} with $(\bar u_{1,\beta}, \bar u_{2,\beta})$, this implies the boundedness of the sequences $\{\lambda_{1,\beta}\}$ and $\{\lambda_{2,\beta}\}$. Moreover, observing that
\[
-\Delta \bar u_{i,\beta} \le \mu_i \bar u_{i,\beta}^3 \qquad \text{in $\R^3$},
\]
through a Br\'ezis-Kato argument we can check that uniform boundedness in $H^1(\R^3,\R^2)$ implies also uniform boundedness in $L^\infty(\R^3,\R^2)$ (see \cite[page 124]{Tath} for a detailed proof, and \cite{BrKa} for the original argument). At this point, the rest of the proof follows directly by the general theory developed in \cite{NoTaTeVe2, SoTaTeZi, SoZi, TaTe2}.

\section{A natural constraint for scalar equations}\label{sec: natural single}

In this and the next sections we deal with the scalar problem \eqref{single}:
\[
\begin{cases}
-\Delta u -\lambda u = f(u) & \text{in $\R^N$} \\
u >0, u \in H^1(\R^N) \\
\int_{\R^N} u^2 = a^2.
\end{cases}
\]
Solutions $(\lambda,u)$ to \eqref{single} are obtained as critical points of the functional $I$, defined in \eqref{functional single intro}, on $S_a$. Let us recall the definition \eqref{def G single} of $G$ in this context:
\[
\begin{split}
G(u) :&=\int_{\R^N} |\nabla u|^2 - \int_{\R^N}  \left( \frac{N}{2} f(u)u-N F(u) \right) \\
& = \int_{\R^N} |\nabla u|^2 - \frac{N}{2} \int_{\R^N} \tilde F(u).
\end{split}
\]
Then we set $\mathcal{M}:= \{ u \in S_a: G(u) = 0\}$. It is known that, thanks to the Pohozaev identity, any solution to \eqref{single} stays in $\mathcal{M}$ (see \cite[Lemma 2.7]{Jea}).
%
The purpose of this section consists in proving a strong version of Theorem \ref{thm: constraint single intro}.

\begin{theorem}\label{thm: constraint single}
Under ($f1$)-($f3$), the set $\mathcal{M}$ is a $\C^1$ manifold, and moreover:
\begin{itemize}
\item[($i$)] If $\{u_{n}\} \subset \mathcal{C}^\infty_c(\R^3) \cap \mathcal{M}$ is a Palais-Smale sequence for $I$ restricted to $\mathcal{M}$ at a certain level $\ell \in \R$, then $\{u_n\}$ is a Palais-Smale sequence for $I$ restricted to $S_a$.

\item[($ii$)] If there exists a Palais-Smale sequence $\{\tilde u_{n}\}$ for $I$ restricted to $\mathcal{M}$ at level $\ell \in \R$, then there exists a possibly different Palais-Smale sequence $\{u_{n}\} \subset \mathcal{C}^\infty_c(\R^3)$ for $I$ restricted to $\mathcal{M}$ at the same level $\ell \in \R$. Moreover $\|\tilde u_n- \tilde u_{i,n}\|_{H^1} \to 0$ as $n \to \infty$.

\item[($iii$)] If there exists a Palais-Smale sequence $\{\tilde u_{n}\}$ for $I$ restricted to $\mathcal{M}$ at level $\ell \in \R$, then there exists a possibly different Palais-Smale sequence $\{u_{n}\} \subset \mathcal{C}^\infty_c(\R^3)$ for $I$ restricted to $S_a$ at the same level $\ell \in \R$. Moreover $\|\tilde u_n- \tilde u_{i,n}\|_{H^1} \to 0$ as $n \to \infty$.

\item[($iv$)] Let $u$ be a critical point of $I$ restricted on $\mathcal{M}$. Then $u$ is a critical point of $I$ restricted on $S_{a}$, and hence a solution to \eqref{complete problem}.
\end{itemize}
\end{theorem}

The proof of this theorem is divided into several intermediate lemmas.

\begin{lemma}\label{lem: single manifold}
The set $\mathcal{M}$ is a $\C^1$ manifold of codimension $1$ in $S_a$, hence a $\C^1$ manifold of codimension $2$ in $H^1(\R^N)$.
\end{lemma}

\begin{proof}
As subset of $H^1(\R^N)$, the constraint $\mathcal{M}$ is defined by the two equations $G(u)=0$, $G_1(u)=0$, where
\[
G_1(u):= a^2 - \int_{\R^3} u^2,
\]
and clearly $G$ and $G_1$ are of class $\C^1$. We have to check that
\[
d(G_1, G): H^1(\R^N) \to \R^2 \quad \text{is surjective}.
\]
If this is not true, $dG(u)$ and $dG_1(u)$ are linearly dependent, i.e. there exist $\nu \in \R$ such that
\[
2 \int_{\R^N} \nabla u \cdot \nabla \varphi - \frac{N}{2}\int_{\R^N} \tilde F'(u) \varphi = 2 \nu \int_{\R^N} u \varphi
\]
for every $\varphi \in H^1(\R^N)$. This means that $u$ is a solution to
\[
\begin{cases}
-\Delta u - \nu u = \frac{N}{4} \tilde F'(u) \\
\int_{\R^3} u^2 = a^2
\end{cases}  \text{in $\R^3$},
\]
Thanks to \cite[Lemma 2.7]{Jea}, we infer that
\[
\int_{\R^N} |\nabla u|^2 = \int_{\R^N} \left( \frac{N^2}{8} \tilde F'(u) u - \frac{N^2}{4} \tilde F(u) \right).
\]
Since $u \in \mathcal{M}$, this gives
\[
\frac{N}{2}\int_{\R^N} \tilde F(u) =  \int_{\R^N} \left( \frac{N^2}{8} \tilde F'(u) u - \frac{N^2}{4} \tilde F(u) \right),
\]
in contradiction with ($f3$) and the fact that $u \in S_a$ (and hence $u \not \equiv 0$ in $\R^N$).
\end{proof}

Let us introduce some notation, similar to that of Section \ref{sec: natural}. For $u \in S_a$ and $s \in \R$, we define
\[
(s \star u)(x):= e^{Ns/2} u(e^s x),
\]
so that $\|s \star u\|_{L^2(\R^N)} = \|u\|_{L^2(\R^N)}$. We also consider
\begin{equation}\label{def Psi single}
\Psi_u(s):= I(s \star u) =   \frac{e^{2s}}{2}\int_{\R^N} |\nabla u|^2 - \frac{1}{e^{Ns}} \int_{\R^N} F(e^{Ns/2}u).
\end{equation}
The study of $\Psi_u$ is the object of the next lemma, for which we refer to \cite[Lemma 2.9]{Jea}.

\begin{lemma}\label{lem: structure M}
For any $u \in S_{a}$, a value $s \in \R$ is a critical point of $\Psi_{u}$ if and only if $s \star u \in \mathcal{M}$. Moreover, for any $u \in S_a$ the function $\Psi_u$ has a unique critical point $s_u$, which is a strict maximum.
\end{lemma}

\begin{remark}\label{on f3 1}
We observe that assumption ($f3$) is used in \cite{Jea} to prove the uniqueness of $s_{u}$.
\end{remark}

\begin{lemma}\label{lem: structure M 2}
The map $u \mapsto s_u$ is of class $\C^1$.
\end{lemma}

\begin{proof}
The value $s_u$ can be found solving the equation  $\Psi'_u(s) = 0$, i.e.
\[
\int_{\R^N} e^{2s} |\nabla u|^2 - \frac{N}{2} e^{-Ns} \tilde F(e^{Ns/2} u) = 0.
\]
Thanks to ($f3$), it is not difficult to check that the assumptions of the implicit function theorem are satisfied.
\end{proof}

In the next lemma we obtain a description of $T_u S_a$ for $u \in \mathcal{M}$, similar to the one in Lemma \ref{lem: splitting}.

\begin{lemma}\label{lem: splitting single}
For any $u \in \mathcal{M} \cap \mathcal{C}^\infty_c(\R^N)$, we have
\[
T_{u} S_{a}  = T_u \mathcal{M} \oplus \R \left. \frac{d}{ds}\right|_{s=0}(s \star u).
\]
\end{lemma}
\begin{proof}
The proof is similar to that of Lemma \ref{lem: splitting}, and hence is sketched. We have to show that
\[
\left. \frac{d}{ds}\right|_{s=0}(s \star u) \in T_u S_a \setminus T_u \mathcal{M}.
\]
%
%
For any $u \in \mathcal{C}^\infty_c(\R^N)$
\[
 \left. \frac{d}{ds}\right|_{s=0}(s \star u)(x) =\frac{N}{2} u(x) + \nabla u(x) \cdot x \in \mathcal{C}^\infty_c(\R^3),
 \]
 and hence it is easy to check that $
 \left. \frac{d}{ds}(s \star u) \right|_{s=0}  \in T_u S_a$. Using the divergence theorem as in Lemma \ref{lem: splitting}, we also obtain
\[
\begin{split}
dG(u) & \left[\left. \frac{d}{ds}\right|_{s=0}(s \star u)\right] = 2 \int_{\R^N} \left[\frac{N}{2} |\nabla u|^2
 + \nabla u \cdot \nabla (\nabla u \cdot x) \right] \\
 & -\frac{N}{2} \int_{\R^N} \tilde F'(u) \left( \frac{N}{2} u + \nabla u \cdot x\right) \\
 & = 2  \int_{\R^N}  |\nabla u|^2 - \frac{N^2}{4} \int_{\R^N} \tilde F'(u) u  + \frac{N^2}{2} \int_{\R^N} \tilde F(u).
\end{split}
\]
Since $u \in \mathcal{M}$, this implies that
\[
\begin{split}
dG(u) & \left[\left. \frac{d}{ds}\right|_{s=0}(s \star u)\right] \\
& = \left( N + \frac{N^2}{2} \right) \int_{\R^N} \tilde F(u) - \frac{N^2}{4} \int_{\R^N} \tilde F'(u) u < 0,
\end{split}
\]
where we used assumptions ($f2$), ($f3$), and the fact that $u \not \equiv 0$.
\end{proof}

\begin{lemma}\label{lem: crit P single}
If $u \in \mathcal{C}^\infty_c(\R^N) \cap \mathcal{M}$, then
\[
dI(u) \left[\left. \frac{d}{ds}\right|_{s=0}(s \star u)\right] = 0.
\]
\end{lemma}

The proof is an easy consequence of the definition of $\mathcal{M}$, see Lemma \ref{lem: crit P} for more details.

\begin{proof}[Conclusion of the proof of Theorem \ref{thm: constraint single}]
We only prove point ($i$). Let $\{u_n\} \subset \mathcal{C}^\infty_c(\R^N) \cap \mathcal{P}$ be a Palais-Smale sequence for $I|_{\mathcal{M}}$. We denote by $(T_{u}S_a)^*$ the dual space to $T_{u}S_a$, and by $\|\cdot\|$ the $H^1(\R^N)$ norm. In view of Lemma \ref{lem: splitting single}, we have:
\begin{align*}
\| dI (u_n) & \|_{(T_{u}S_a)^*}  = \sup\left\{|dI(u_n)[\varphi]| : \varphi \in T_{u}S_a, \  \|\mf{\varphi}\| \le 1  \right\} \\
& = \sup\left\{ |dI(u_n)[\phi] + dI(u_n)[\psi] |   \left| \begin{array}{l} \varphi = \phi +\psi, \|\varphi\| \le 1 \\
\phi \in T_{u}\mathcal{M},  \ \mf{\psi} \in \R\left(\left. \frac{d}{ds}\right|_{s=0} (s \star u_n)\right) \end{array}\right.\right\}.
\end{align*}
Now $dI(u_n)[\psi]=0$ by Lemma \ref{lem: crit P single}, and hence
\begin{align*}
\| dI (u_n) & \|_{(T_u(S_a))^*}  = \sup\left\{|dI(u_n)[\phi]| : \phi \in T_{u}\mathcal{P}, \
\|\phi\| \le 1 \right\} \\
& = \| dI (u_n)  \|_{(T_{u}\mathcal{M})^*} \to 0
\end{align*}
as $n \to \infty$, since $\{u_n\}$ is a Palais-Smale sequence for $I$ restricted to $\mathcal{M}$.
\end{proof}

\section{Existence and multiplicity of solutions to \eqref{single}}\label{sec: existence single}

This section is devoted to the proof of Theorem \ref{thm: main single}. We are interested in the existence or radial solutions, and hence throughout this section we will always work in $S_a^r$. This simplifies some compactness issues. As a consequence of Theorem \ref{thm: constraint single}, the existence of solutions to \eqref{single} reduces to the existence of critical points for $I$ restricted to $\mathcal{M}$. The main advantage is that, in contrast to $I$ restricted on $S_a$, the functional $I$ restricted to $\mathcal{M}$ satisfies the Palais-Smale condition, and is bounded from below. Thus, the Lusternik-Schnirelman theory yields infinitely many critical points. This idea is rigorously developed in what follows.

We denote by $\cat_{\Z/2}(\mathcal{M})$ the equivariant Lusternik-Schnirelman category of $\mathcal{M}$ with respect to the antipodal action of $\Z/2$, and by $\gen(\mathcal{M})$ the Krasnoselskii genus of $\mathcal{M}$. For the definitions and the properties of $\cat_{\Z/2}$ and $\gen$ we refer to \cite[Section 2]{Babook} (there it is considered a much more general setting with respect to the one considered here; an easier reference for the genus is \cite{AmbMal}).

Notice that both $I$, $G$ and $G_1$ are even functionals, and hence the problem is invariant under the action of $\Z/2$.

\begin{lemma}\label{lem: LS cat}
It results that $\cat_{\Z/2}(\mathcal{M}) = +\infty$.
\end{lemma}
\begin{proof}
It is well known that $\cat_{\Z/2}(\mathcal{M}) \ge \gen(\mathcal{M})$, see for instance \cite[Proposition 2.10]{Babook}. Therefore, we can prove that $\gen(\mathcal{M})= +\infty$. Let $V \subset H^1(\R^N)$ with $\dim V = n$, and let $SV:= V \cap S_a^r$. Notice that $\gen(SV) = \dim V = n$ (this follows for instance by \cite[Theorem 10.5]{AmbMal}). We show that there exists a map $\psi: SV \to \mathcal{M}$ continuous and odd, whence by \cite[Lemma 10.4]{AmbMal} we deduce that $\gen(\mathcal{M}) \ge \gen(SV) =n$; since $n$ is arbitrary, the thesis follows.

The explicit expression of $I(s \star u)$ (see \eqref{def Psi single}) and the oddness of $f$ ensure that the map $SV \ni u \mapsto s_u \in \R$ is even: $s_u = s_{-u}$. It is also continuous by Lemma \ref{lem: structure M 2}. The map $\psi(u) = s_u \star u$ is then odd because
\[
\psi(-u) = s_{-u} \star (-u) = - s_u \star u = - \psi(u),
\]
and it is also continuous due to Lemma \ref{lem: s strong convergence}.
\end{proof}

Now we describe the properties of $I$ on $\mathcal{M}$. We shall use many times the following inequalities, which can be easily proved using assumptions ($f1$) and ($f2$): for every $t \in \R$ and $s \ge 0$ there holds
\begin{equation}\label{f1 e f2}
\begin{cases}
s^\beta F(t) \le F(ts) \le s^\alpha F(t) & \text{if $s \le 1$} \\
s^\alpha F(t) \le F(ts) \le s^\beta F(t) & \text{if $s \ge 1$}.
\end{cases}
\end{equation}
We also recall the Gagliardo-Nirenberg inequality: There exists a constant $S$ depending on $N$ and on $r \in (2,2^*)$ such that
\begin{equation}\label{Gag-Nir}
\|u\|_{L^r} \le S \|u\|_{L^2}^{1-\gamma} \|\nabla u\|_{L^2}^{\gamma} \qquad \text{for all }u \in H^1(\R^N);
\end{equation}
here $\gamma= N\left( \frac12 - \frac{1}{r}\right)$.

\begin{lemma}\label{lem: non-degeneracy}
There exists $\delta>0$ such that $\|u\|_{\D^{1,2}} \ge \delta$ for every $u \in \mathcal{M}$.
\end{lemma}

\begin{proof}
Since $F(s) \ge 0$ for every $s \in \R$ and by ($f2$), for $u \in \mathcal{M}$ we have
\begin{align*}
\int_{\R^3} |\nabla u|^2 &\le \frac{N}{2}\int_{\R^3} f(u)u \le  \frac{N\beta}{2} \int_{\R^3} F(u) \\
& \le  \frac{N\beta}{2}\int_{\{|u| \le 1\}} F(u) + \frac{N\beta}{2}\int_{\{|u| \ge 1\}} F(u) \\
& \le \frac{N\beta}{2} \int_{\{|u| \le 1\}} F(1) |u|^\alpha  +  \frac{N\beta}{2}\int_{\{|u| \ge 1\}} F(1)|u|^\beta \\
& \le C \int_{\R^3} |u|^\alpha + |u|^\beta,
\end{align*}
where we used \eqref{f1 e f2}. To estimate the right hand side, we apply \eqref{Gag-Nir} with $r=\alpha$ and $r=\beta$, obtaining
\[
\|\nabla u\|_{L^2}^2 \le C \|\nabla u\|_{L^2}^{\frac{N}{2}(\alpha-2)} + C \|\nabla u\|_{L^2}^{\frac{N}{2}(\beta-2)}.
\]
Now due to ($f2$) we know that both $N(\alpha-2)/2$ and $N(\beta-2)/2$ are strictly larger than $2$, and hence the lemma follows.
\end{proof}

\begin{lemma}\label{lem: bdd below}
The functional $I$ restricted to $\mathcal{M}$ is coercive and bounded from below by a positive constant.
\end{lemma}

\begin{proof}
By ($f2$), we infer that for any $u \in \mathcal{M}$
\begin{equation}\label{2004}
\int_{\R^3} |\nabla u|^2 \le \frac{N}{2}\int_{\R^N} f(u) u \le \frac{N\beta}{2} \int_{\R^N} F(u).
\end{equation}
Therefore, using again ($f2$)
\begin{align*}
I(u) & =  \frac{N}{4} \int_{\R^N} f(u) u- \left(\frac{N+2}{2}\right) \int_{\R^N} F(u)
 = \frac{N}{4}  \int_{\R^N} \left( f(u) u - \left(2+\frac{4}{N}\right) F(u) \right) \\
& \ge \frac{N}{4}   \left( \alpha-2-\frac{4}{N}\right) \int_{\R^N} F(u)
 \ge \frac{1}{2\beta}  \left( \alpha-2-\frac{4}{N}\right) \int_{\R^N} |\nabla u|^2
\end{align*}
for any $u \in \mathcal{M}$. Now Lemma~\ref{lem: bdd below} follows from Lemma \ref{lem: non-degeneracy}.
\end{proof}

\begin{lemma}\label{PS on M}
The Palais-Smale condition is satisfied by $I$ constrained to $\mathcal{M}$.
\end{lemma}

\begin{proof}
Let $\{\tilde u_n\} \subset \mathcal{M}$ be a Palais-Smale sequence for $I|_{\mathcal{M}}$ at some level $c \in \R$ (notice that automatically $c>0$ by Lemma \ref{lem: bdd below}), and let $\{u_n\} \subset \mathcal{M} \cap \C^\infty_c(\R^N)$ be the possibly different Palais-Smale sequence given by Theorem \ref{thm: constraint single}-($ii$). It is sufficient to show that $\{u_n\}$ converges strongly in $H^1(\R^N)$ to some limit, up to a subsequence.

By Lemma \ref{lem: bdd below} $\{u_n\}$ is bounded, and hence up to a subsequence $u_n \wc u$ weakly in $H^1(\R^3)$, for a suitable $u \in H^1(\R^N)$. Moreover, due to Theorem \ref{thm: constraint single} and the Lagrange multipliers rule (see also \cite[Lemma 2.5]{Jea} for more details), we have
\[
\int_{\R^N} (\nabla u_n \cdot \nabla \varphi - f(u_n) \varphi -\lambda_n u_n \varphi ) = o(1) \|\varphi\|_{H^1}
\]
for every $\varphi \in H^1(\R^N)$, where $o(1) \to 0$ as $n \to \infty$ and $\lambda_n \in \R$. Taking $\varphi=u_n$ and recalling the definition of $\mathcal{M}$, we deduce that
\begin{align*}
\lambda_n a^2 & = \int_{\R^N} (|\nabla u_n|^2 - f(u_n) u_n) + o(1)\\
& \le \int_{\R^N} \left( \left(\frac{N-2}{2} \right) f(u_n) u_n - N F(u_n) \right) + o(1).
\end{align*}
Let $N \ge 3$; using assumption ($f2$), estimate \eqref{2004} and Lemma \ref{lem: non-degeneracy}, the previous computation gives
\begin{align*}
\lambda_n a^2  &\le \int_{\R^N} \left( \frac{N-2}{2N} \right) \left(\beta- \frac{2N}{N-2}\right) F(u_n) + o(1) \\
& \le C \int_{\R^N} |\nabla u_n|^2 \le -C <0.
\end{align*}
If $N=2$, the same conclusion follows using simply estimate \eqref{2004} and Lemma \ref{lem: non-degeneracy}.

Notice also that $\{\lambda_n\}$ is bounded (since $\{u_n\}$ is), and hence up to a subsequence $\lambda_n \to \lambda <0$.

The conclusion of the proof follows from now on exactly as in \cite[Section 2.4]{Jea}.
\end{proof}

\begin{proof}[Proof of Theorem \ref{thm: main single}]
Due to Lemmas \ref{lem: bdd below} and \ref{PS on M}, we can apply the Lusternik-Schnirelman Theorem 2.19 in \cite{Babook}; this, together with Lemma \ref{lem: LS cat}, completes the proof of existence and multiplicity. We also observe that the minimizer for $I|_{\mathcal{M}}$ can be taken positive, because $u\in\mathcal{M}$ implies $|u|\in\mathcal{M}$ and $I(u)=I(|u|)$.
\end{proof}

\begin{remark}
Theorem 2.19 in \cite{Babook} is stated for $\C^1$ functionals on $\C^{2-}$ manifolds, while under our assumption $\mathcal{M}$ is merely $\C^1$. This is not a problem, as observed in \cite[page 21]{Babook}, since the Szulkin's approach developed in \cite{SZ} permits to replace the $\C^{2-}$ assumption in \cite{Babook} with simple $\C^1$ regularity.
\end{remark}


\end{document}